\documentclass[11pt,reqno,centertags]{amsart}
\usepackage{asymptote}
\usepackage{cite}
\usepackage{hyperref}
\hypersetup{
    colorlinks,
    citecolor=blue,
    filecolor=blue,
    linkcolor=blue,
    urlcolor=blue
}
\usepackage[english]{babel}
\usepackage{amsmath, amsfonts,amssymb}
\usepackage{enumerate}
\usepackage{graphicx}
\newcommand{\eps}{\varepsilon}
\newcommand{\nl}{\mathcal S_-}
\newcommand{\n}{\mathcal S}
\newcommand{\np}{\mathcal S_+}
\renewcommand{\v}{\mathcal{N}}
\newcommand{\vl}{\mathcal{N_-}}
\newcommand{\vp}{\mathcal{N_+}}
\renewcommand{\l}{\mathcal{W}}
\renewcommand{\ln}{\mathcal{W_-}}
\newcommand{\lv}{\mathcal{W_+}}
\newcommand{\p}{\mathcal{E}}
\newcommand{\pn}{\mathcal{E_-}}
\newcommand{\pv}{\mathcal{E_+}}
\newcommand{\nz}{n\mathbb{Z}^2}

\newcommand{\m}{P}
\newcommand{\tm}{\widehat{\m}}
\newcommand{\zz}{\mathbb{Z}^2}
\newcommand{\rr}{\mathbb{R}^2}
\newcommand{\real}{\mathbb{R}}
\newcommand{\z}{\mathbb{Z}}

\newcommand{\set}[2]{\{#1 \colon #2 \}}

\newcommand{\diag}{\mathop{\mathrm{diag}}}

\newcommand{\floor}[1]{\lfloor #1\rfloor}
\newcommand{\Floor}[1]{\left\lfloor #1\right\rfloor}
\newcommand{\ceil}[1]{\lceil #1\rceil}
\newcommand{\Ceil}[1]{\left\lceil #1\right\rceil}

\usepackage{amsthm}
\newtheorem{theorem}{Theorem}[section]
\newtheorem{lemma}[theorem]{Lemma}
\newtheorem{proposition}[theorem]{Proposition}
\newtheorem{corollary}[theorem]{Corollary}

\newtheorem*{subtheorem-a}{Sub-Theorem A}
\newtheorem*{subtheorem-b}{Sub-Theorem B}
\newtheorem*{subtheorem-c}{Sub-Theorem C}
\theoremstyle{definition}
\newtheorem{definition}[theorem]{Definition}

\theoremstyle{remark}
\newtheorem{remark}[theorem]{Remark}

\numberwithin{equation}{section}
\begin{document}

\title[Sublattice-free lattice polygons]{Bounds on the number of vertices of 
sublattice-free lattice polygons}
\author[N.~Bliznyakov]{Nikolai Bliznyakov}
\address{Faculty of Mathematics, Voronezh State University, 1 Universitetskaya 
pl., Voronezh, 394006, Russia}{}
\email{bliznyakov@vsu.ru}
\author[S.~Kondratyev]{Stanislav Kondratyev}
\address[S.~Kondratyev]{CMUC, Department of
Mathematics, University of Coimbra, 3001-501 Coimbra, Portugal}{}
\email{kondratyev@mat.uc.pt}




\begin{abstract}
In this paper we establish bounds on the number of vertices for a few classes 
of convex sublattice-free lattice polygons.  The bounds are essential for 
proving the formula for the critical number of vertices of a lattice polygon 
that ensures the existence of a sublattice point in the polygon.  To obtain 
the bounds, we use relations between the number of edges of lattice broken 
lines and the coordinates of their endpoints.
\end{abstract}


\maketitle

\section{Introduction}

Papers studying geometric and combinatorial properties of convex lattice 
polytopes and polygons are quite numerous: \cite{Sco88, AZ95, Sim90, Lov89, 
BCCZ10, MD11, Ave13, Hen83, LZ91, Pik01, AKN15, Rab89, Rab90} to cite a few; 
see also the monographs \cite{GL87, EGH89, Gru07, BR09}.  Our present study is 
motivated by our paper~\cite{BKa}.

Remember that a \emph{lattice} in $\rr$ spanned by a given linearly 
independent system of two vectors is the set of of integral linear 
combinations of the system.  The system itself is called the \emph{basis} of 
the lattice.  The \emph{integer lattice}~$\zz$ is the lattice spanned by the 
standard basis of~$\rr$.  The points of~$\zz$ are called \emph{integer 
points}.  We are primarily interested in the integer lattice and the ones 
contained in it, i.~e. its sublattices.

By a convex polygon we understand the convex hull of a finite set of points 
in~$\rr$ that has nonempty interior.  We assume that the reader is familiar 
with basic terminology such as vertices and edges, see~\cite{Gru03, Zie95} for 
reference.  As we never consider nonconvex polygons, we occasionally drop the 
word `convex'.  If a polygon has~$N$ vertices, we refer to it as an $N$-gon. 
An \emph{integer polygon}, or a \emph{lattice polygon} is a polygon, whose 
vertices are integer points.  Generally, we prefer the former term, as it is 
unambiguous in contexts where a few lattices are considered simultaneously.

It was noted in~\cite{Bli00} that any convex integer pentagon contains a point 
of the lattice~$2\zz = (2\z) \times (2\z)$.  The paper~\cite{BKa} raises the 
following question: given a sublattice~$\Lambda$ of the integer lattice~$\zz$, 
what is the smallest number of vertices of an integer polygon that ensures 
that the polygon contains at least one point of~$\Lambda$?  The answer is the 
Main Theorem of~\cite{BKa}.  To state it, we recall that any sublattice 
of~$\zz$ is characterised by two positive integers called invariant factors 
(see the definition in Section~\ref{sec:prelim:geom}).  Let $\delta$, $n$ be 
the invariant factors of~$\Lambda$ and set
\begin{equation*}
\nu(\Lambda) = \nu(\delta, n) =2n+2\min\{\delta,3\}-3
.
\end{equation*}
\begin{theorem}[Main Theorem of \cite{BKa}]
\label{th:mt}
Let~$\Lambda$ be a proper sublattice of $\zz$ having the invariant factors 
$\delta$, $n$.  Then any convex integer polygon with~$\nu(\Lambda)$ vertices 
contains a point of $\Lambda$.
\end{theorem}
The constant~$\nu(\Lambda)$ in Theorem~\ref{th:mt} cannot be improved.

Remarkably, Theorem~\ref{th:mt} does not allow for straightforward 
generalisations to higher dimensions.

The paper~\cite{BKa} does not give a full proof of Theorem~\ref{th:mt}.  It is 
shown that the the proof can be reduced to obtaining upper bounds on the 
number of vertices of integer polygons free of points of the lattice $\nz = (n 
\z) \times (n \z)$.  The paper~\cite{BKa} introduces the following 
classification of such polygons.

We say that a line or a segment \emph{splits} a polygon, if it divides the 
polygon into two parts with nonempty interior.  Hereafter $[\mathbf a, \mathbf 
b]$ denotes the segment with the endpoints~$\mathbf a$ and~$\mathbf b$.

\begin{figure}
\label{fig:types}
\includegraphics{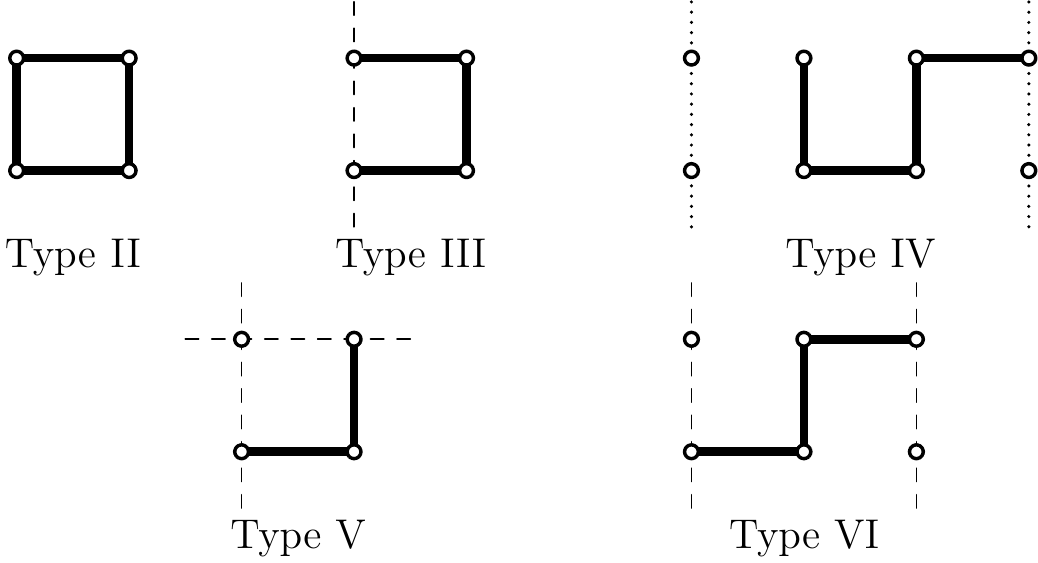}
\caption{Definition~\ref{def:types} introduces the types of polygons in terms 
of intersection with segments and lines.  Here thick segments split polygons 
of the specified type, thin lines do not split them, and dotted lines have no 
common points with them.}
\end{figure}

Let $\m$ be an integer polygon free from points of~$\nz$, where $n \ge 2$ is 
an integer.

\begin{definition}
\label{def:types}
We say that $\m$ is a
\begin{itemize}
\item
\emph{type~I$_n$ polygon}, if no line of the form $x_1 = jn$ or $x_1 = jn$ 
where $j \in \z$, splits $\m$, or, equivalently, if $\m$ lies in a slab of the 
form $jn \le x_1 \le (j+1)n$ or $jn \le x_2 \le (j+1)n$, where $j \in 
\z$;
\item
\emph{type~II$_n$ polygon}, if each of the segments $[\mathbf{0},(n,0)]$, 
$[(n,0),(n,n)]$, $[(0,n),(n,n)]$, and $[\mathbf{0},(0,n)]$ splits~$\m$;
\item
\emph{type~III$_n$ polygon}, if each of the segments $[\mathbf{0},(n,0)]$, 
$[(n,0),(n,n)]$, and $[(n,n),(0,n)]$ splits~$\m$, and the line $x_1 = 0$ does 
not split $\m$;
\item
\emph{type~IV$_n$ polygon}, if each of the segments $[\mathbf{0},(0,n)]$, 
$[\mathbf 0, (n, 0)]$, $[(n, 0), (n, n)]$, and $[(n, n), (2n, n)]$ splits~$\m$ 
and~$\m$ has no common points with the lines $x_1 = -n$ and $x_n = 2n$;
\item
\emph{type~V$_n$ polygon}, if each of the segments $[\mathbf 0, (-n ,0)]$ and 
$[\mathbf 0, (0, n)]$ splits $\m$ and the lines $x_1 = -n$ and $x_2 = n$ do 
not split~$\m$;
\item
\emph{type~VI$_n$ polygon}, if each of the segments $[\mathbf 0, (-n ,0)]$, 
$[\mathbf 0, (0, n)]$, and $[(0, n), (n, n)]$ splits~$\m$, and the lines $x_1 
= \pm n$ do not split~$\m$.
\end{itemize}
\end{definition}

The polygon types are illustrated on Figure~\ref{fig:types}.

The following theorem is proved in~\cite{BKa}.

\begin{theorem}
\label{th:types}
Suppose that an integer polygon $\m$ is free of points of the lattice $\nz$, 
where $n \in \z$, $n\ge2$; then there exists an affine transformation 
$\varphi$ of $\rr$ preserving $\nz$ such that $\varphi(\m)$ is a polygon of 
one of the types I$_n$--VI$_n$.
\end{theorem}

The hard part of Theorem~\ref{th:mt} is encapsulated in the following 
assertion.

\begin{theorem}[Sub-Theorem~C of~\cite{BKa}]
\label{subtheorem-c}
Let $\m$ be a convex integer $N$-gon of one of the types~I$_n$--VI$_n$, where 
$n$ is an integer, $n \ge 3$.  Then:
\begin{enumerate}[(i)]
\item
the following inequality holds:
$$
N\le2n+2
;
$$
\item
if the vertices of $\m$ belong to a lattice with invariant factors $(1, n/2)$, 
then
$$N\le2n;$$
\item
if the vertices of $\m$ belong to a lattice with invariant factors $(1, n)$, 
then
$$N\le2n-2.$$
\end{enumerate}
\end{theorem}

It is shown in~\cite{BKa} that Theorem~\ref{subtheorem-c} together with other 
results of that paper imply Theorem~\ref{th:mt}.  In~\cite{BKa} 
Theorem~\ref{subtheorem-c} is proved for type~I and~II polygons.  The present 
paper aims to prove this theorem for the rest cases, systematically applying 
the approach developed in~\cite{BKa}.  Thus, the proof of Theorem~\ref{th:mt} 
will be also completed.

The core of the method is constituted by a few statements about certain 
classes of broken lines. These statements provide relations between the number 
of edges of the broken lines and the coordinates of their endpoints.  Our 
strategy in dealing with specific types of polygons is to translate geometric 
constraints imposed on a polygon into Diaphantine inequalities relating the 
numbers of vertices of certaing broken lines contained in the boundary of the 
polygon, the parameters of its bounding box and, possibly, auxiliary integral 
parameters.  Then we analyse the inequalities trying to obtain estimates on 
the number of vertices.  This can prove rather technical due to the number of 
parameters and nonlinearities.

The rest of the paper is organised as follows.

In Section~\ref{sec:prelim} we collect familiar facts concerning the geometry 
of lattices and convex polygons as well as results of~\cite{BKa} useful for 
estimating the number of edges of lattice broken lines and polygons.

In Sections~\ref{sec:iii} and~\ref{sec:iv} we prove Theorem~\ref{subtheorem-c} 
for type~III and~IV polygons, respectively, applying the method described 
above.  In the case of type~III polygons the transition from geometric 
constraints to Diaphantine inequalities is fairly straightforward, but the 
analysis of the inequalities is rather involved.  Type~IV polygons are 
somewhat more technical from the geometric point of view.

In Section~\ref{sec:v} we consider type~V polygons.  We show that iterating 
so-called lift transitions in combination with certain affine transformations 
it is always possible to map any type~V polygon either onto a type~III polygon 
or onto a polygon lying in a certain triangle (we call them type~Va polygons).  
As the case of type~III polygons is know, it clearly suffices to consider 
type~Va polygons instead of type~V polygons.  We establish a few bounds on the 
number of vertices of type~Va polygons using various algebraic and geometric 
tricks; however, at this point we are unable to obtain all the estimates 
required by Theorem~\ref{subtheorem-c}.  We revisit type Va polygons in 
Section~\ref{sec:last}.

In Section~\ref{sec:vi} we reuse the lift transformations introduced in the 
previous section and show that any type~VI polygon can be mapped onto a 
polygon of another type.  However, type~V is not excluded, so at this point we 
are unable to prove Theorem~\ref{subtheorem-c} for type~VI polygons.  The 
theorem is proved in the next section.

In Section~\ref{sec:last} we finally prove the missing inequality for type~Va 
polygons, which completes the proofs of Theorems~\ref{subtheorem-c} and, 
eventually,~\ref{th:mt}.  What enables us to carry out this proof is a bound 
on the number of vertices of an arbitrary integer polygon free of points 
of~$\nz$, obtained as a combination of established estimates and 
Theorem~\ref{th:types}.

\section{Preliminaries}
\label{sec:prelim}

\subsection{Lattices and polygons}
\label{sec:prelim:geom}

In this section we collect a few definitions and facts concerning the geometry 
of the integer lattice and convex polytops.  For reference, see \cite{Cas12, 
GL87, Gru07, BR09}.

A lattice $\Lambda \subset \zz$ is spanned by the columns of a matrix $A = 
(a_{ij}) \in GL_2(\z)$ if and only if $\Lambda = A \zz = \set{A\mathbf u}{ 
\mathbf u\in\zz}$. Given $\Lambda$, the matrix $A$ is not uniquely defined.  
However, the numbers
\begin{equation*}
\delta = \gcd(a_{ij}),\ n = |\det A|/\delta
\end{equation*}
are independent of $A$.  They are called \emph{invariant factors} of 
$\Lambda$, and the pair $(\delta, n)$ is \emph{the invariant factor sequence 
of} $\Lambda$ (see \cite{Nor12}).  The product of the invariant factors equals 
the determinant of~$A$; it is called the \emph{determinant} of the lattice and 
denoted~$\det \Lambda$.  Clearly, proper sublattices of~$\zz$ (i.~e.\ the ones 
that do not coincide with~$\zz$) have determinants $\ge 2$.

For brevity, we write that $\Lambda$ is a $(\delta, n)$-lattice if it is a 
sublattice of~$\zz$ with invariant factor sequence $(\delta, n)$.

As an example, the lattice $n\zz = \set{(nu_1, nu_2)}{u_1, u_2 \in \zz}$, 
where $n$ is a positive integer, has invariant factor sequence $(n, n)$, and 
the lattice $\delta \z \times n \zz = \set{(\delta u_1, nu_2)}{u_1, u_2 \in 
\zz}$, where $\delta$ and $n$ are positive integers and $\delta$ divides $n$, 
has invariant factor sequence $(\delta, n)$.

If a point belongs to a lattice~$\Lambda$, we call it a $\Lambda$-point.

A matrix $A \in M_2(\z)$ is called \emph{unimodular}, if $\det A = \pm 1$.  

\begin{proposition}
\label{pr:th1-1}
Let $(\mathbf f_1, \mathbf f_2)$ be a basis of a lattice~$\Lambda$; then the 
vectors $a_{i1} \mathbf f_j + a_{i2} \mathbf f_2$, where $i = 1, 2$, form a 
basis of $\Lambda$ if and only if the matrix $(a_{ij})$ is unimodular.
\end{proposition}

A linear transformation of the plane is called a linear automorphism of a 
lattice if it maps the lattice onto itself.  A linear transformation is an 
automorphism of a lattice if and only if it maps some (hence, any) basis of 
the lattice onto another basis.  Clearly, linear automorphism of a lattice 
form a group.

Given a matrix $A \in M_2(\real)$, the transformation $\mathbf x \mapsto 
A\mathbf x$ is an automorphism of $\zz$ if and only if the matrix~$A$ is 
unimodular.  We call such transformation \emph{unimodular}.  What is more, for 
any positive integer $n$, the linear automorphisms of $\nz$ are exactly the 
unimodular transformations.

If $\Lambda$ is a sublattice of $\zz$ and $A$ is a unimodular transformation, 
the image $A\Lambda$ is a lattice with the same invariant factors as 
$\Lambda$.

The following proposition is a geometric version of the Smith normal form of 
integral matrices~\cite{Nor12}.

\begin{proposition}
\label{pr:snf}
For any sublattice of $\zz$ with invariant factors $(\delta, n)$ there exists 
a unimodular transformation mapping it onto the lattice $\delta \z \times n 
\z$.
\end{proposition}

An \emph{affine frame} of a lattice~$\Lambda$ is a pair $(\mathbf o; \mathbf 
f_1, \mathbf f_2)$ consisting of a point $\mathbf o \in \Lambda$ and a basis 
$(\mathbf f_1, \mathbf f_2)$ of~$\Lambda$.  An \emph{integer frame} is an 
affine frame of~$\zz$.

An \emph{affine automorphism} of a lattice $\Lambda$ is an affine 
transformation of $\rr$ mapping $\Lambda$ onto itself.  It is not hard to see 
that given $A \in M_2(\real)$ and $\mathbf b \in \rr$, the mapping $\mathbf x 
\mapsto A \mathbf x + \mathbf b$ is an affine automorphism of $\Lambda$ if and 
only if $\mathbf x \mapsto A \mathbf x$ is an automorphism of $\Lambda$ and 
$\mathbf b \in \Lambda$.  In particular, affine automorphisms of $\nz$, where 
$n$ is a positive integer, are exactly the transformations of the form 
$\mathbf x \mapsto A \mathbf x + \mathbf b$, where $A$ is unimodular and 
$\mathbf b \in \nz$.

Of course, if $\m$ is a convex integer $N$-gon and $\varphi$ is an affine 
automorphism of~$\zz$, the image $\varphi(\m)$ is still a convex integer 
$N$-gon.  Obviously, is $\m$ is free from points of a lattice $\Lambda$, then 
so is its image under any affine automorphism of $\Lambda$.

Following~\cite{BKa}, we introduce the following definition.

Let $\Lambda$ be a sublattice of~$\zz$ and $(\mathbf f_1, \mathbf f_2)$ be a 
basis of~$\zz$.  Clearly, $\{u \in \z \colon u \mathbf f_1 \in \Lambda\}$ is a 
subgroup of~$\z$.  It is generated by a positive integer, which we call the 
\emph{large $\mathbf f_1$-step} of~$\Lambda$ with respect to $(\mathbf f_1, 
\mathbf f_2)$.  Further, $\{u_1 \in \z \colon \exists\ u_2 \in z,\ u_1 f_1 + 
u_2 f_2 \in \Lambda\}$ is a subgroup of~$\z$, too.  We call its positive 
generator the \emph{small $\mathbf f_1$-step} of~$\Lambda$.  Alternatively, 
the small $\mathbf f_1$-step can be defined as the largest $s$ such that all 
the points of~$\Lambda$ lie on the lines $\{ks\mathbf f_1 + t\mathbf f_2\}$, 
$k \in \z$.  Obviously, the small step is smaller then the large step.  We can 
define the large and small $\mathbf f_2$-steps of $\Lambda$ with respect 
to~$(\mathbf f_1, \mathbf f_2)$ in the same way.

In what follows we consider small and large steps of lattices with respect to 
bases made up of the vectors $\pm \mathbf e_1$, $\pm \mathbf e_2$, where 
$(\mathbf e_1, \mathbf e_2)$ is the standard basis of~$\rr$, and we usually 
omit the reference to the basis when there is no ambiguity.

We note two simple properties of the steps.

\begin{proposition}
\label{pr:th1-11}
Let $\Lambda$ be a sublattice of~$\zz$ and $(\mathbf f_1, \mathbf f_2)$ be a 
basis of~$\zz$.  Then the product of the small $\mathbf f_1$-step and the 
large $\mathbf f_2$-step of~$\Lambda$ equals the determinant of~$\Lambda$.
\end{proposition}

\begin{proposition}
\label{pr:th1-12}
Let $(\mathbf f_1, \mathbf f_2)$ be a basis of~$\zz$, $\Lambda$ be a 
sublattice of~$\zz$ such that the small $\mathbf f_1$-step of~$\Lambda$ is~1, 
and $R$ be a complete residue system modulo~$\det \Lambda$.  Then there exists 
a unique $r \in R$ such that $(\mathbf f_1 + r \mathbf f_2, (\det \Lambda) 
\mathbf f_2)$ is a basis of~$\Lambda$.
\end{proposition}

The following Lemma and its corollary are geometrically obvious.  An 
application of Helly's theorem provides an immediate proof of the proposition.

\begin{lemma}
\label{Helly}
If a convex polygon has common points with each of the four angles formed by 
intersecting lines, it contains the intersection point.
\end{lemma}
\begin{corollary}
\label{cor:Helly}
If a convex polygon has common points with both sides of one of the vertical 
angles and does not contain its vertex, it has no common points with the other 
vertical angle.
\end{corollary}

We always denote the vectors of the standard basis of~$\rr$ by $\mathbf e_1 = 
(1, 0)$ and $\mathbf e_2 = (0, 1)$ and the standard coordinates in~$\rr$ by 
$x_1$, $x_2$.  We also use usual notations $\floor{\cdot}$ for the floor 
function, $\ceil{\cdot}$ for the ceiling function, and $|\cdot|$ for the 
cardinality of a finite set.

\subsection{Slopes}

This section summarises the results of~\cite{BKa} about a class of broken 
lines called slopes.  These are our main tool for obtaining bounds on the 
number of vertices of polygons.  The proofs can be found in~\cite{BKa}.

Let $(\mathbf{f}_1,\mathbf{f}_2)$ be a basis of $\rr$, and let $\mathbf v_0, 
\mathbf v_1 , \dots, \mathbf v_N$ ($N \ge 0$) be a finite sequence of points 
on the plane.  If $N\ge1$, set
\begin{equation}\label{eq1-2-3}
\mathbf{v}_i+\mathbf{v}_{i-1} =
\mathbf{a}_i=a_{i1}\mathbf{f}_1+a_{i2}\mathbf{f}_2\qquad
(i=1,\ldots,N)
.
\end{equation}
If
\begin{equation}\label{eq3-3}
a_{i1}>0,\ a_{i2}<0\qquad (i=1,\ldots,N)
\end{equation}
and
\begin{equation}\label{eq4-3}
\begin{vmatrix}
a_{i1}&a_{i+1,1}\\
a_{i2}&a_{i+1,2}
\end{vmatrix}
>0
\qquad (i=1,\ldots,N-1)
,
\end{equation}
we say that the union $Q$ of the segments $[\mathbf v_{0}, \mathbf v_1]$, 
$[\mathbf v_{1}, \mathbf v_2]$, \dots, $[\mathbf v_{N-1}, \mathbf v_N]$ is a 
\emph{slope} with respect to the basis $(\mathbf f_1,\mathbf f_2)$.  These 
segments are called the \emph{edges} of the slope, and the points $\mathbf 
v_0$, $\mathbf v_1$, \dots, $\mathbf v_N$, its \emph{vertices}, $\mathbf v_0$ 
and $\mathbf v_N$ being the \emph{endpoints}.  If $N = 1$, we call the segment 
$[\mathbf v_0, \mathbf v_1]$ a slope if \eqref{eq3-3} holds, and if $N = 0$, 
we still call the one-point set $\{\mathbf v_0\}$ a slope.  If all the 
vertices of $Q$ belong to a lattice $\Gamma$, we call it a 
$\Gamma$-\emph{slope}.  A $\zz$-slope is called \emph{integer}, and it is the 
only kind of slopes we are interested in.

If $Q$ is a slope with respect to a basis $(\mathbf f_1, \mathbf f_2)$, then 
it is a slope with respect to the basis $(\mathbf f_2, \mathbf f_1)$ as well.

\begin{proposition}
\label{pr:slp}
Let $(\mathbf f_1, \mathbf f_2)$ be a basis of~$\zz$ and~$\mathbf v$ 
and~$\mathbf w$ be the endpoints of an integer slope (with respect to 
$(\mathbf f_1, \mathbf f_2)$) having~$N$ edges.  Let
\begin{equation*}
\mathbf w - \mathbf v = b_1 \mathbf f_1 + b_2 \mathbf f_2
.
\end{equation*}
Then there exists an integer $s$ such that
\begin{gather}
2N\le |b_1| + s,
\label{eq:slp1}
\\
|b_2| \ge\frac{s(s+1)}{2},
\label{eq:slp2}
\\
0\le s\le N.
\label{eq:slp3}
\end{gather}
If the vertices of the slope belong to a lattice with small $\mathbf f_1$-step 
greater then~1, one can take $s = 0$, so that
\begin{equation}
\label{eq:slp4}
2N \le |b_1|
.
\end{equation}
If the vertices of the slope belong to a lattice having the basis 
$(\mathbf{f}_1-a\mathbf{f}_2,m\mathbf{f}_2)$, where $1\le
a\le m$, then~\eqref{eq:slp2} can be replaced by
\begin{equation}
\label{eq:slp5}
|b_2| \ge\frac{2a+(s-1)m}{2}s
.
\end{equation}
\end{proposition}

Let $(\mathbf o;\mathbf{f}_1,\mathbf{f}_2)$ be an affine frame of~$\zz$ 
and~$Q$ be a slope with respect to $(\mathbf f_1, \mathbf f_2)$.

\begin{definition}
We say that the frame $(\mathbf o;\mathbf{f}_1,\mathbf{f}_2)$ \emph{splits} 
the slope~$Q$, if
\begin{enumerate}
\item
one endpoint $\mathbf v = \mathbf o + v_1 \mathbf f_1 + v_2 \mathbf f_2$ of 
$Q$ satisfies
\begin{equation}
\label{eq:sf1}
v_1 < 0, \ v_2 > 0,
\end{equation}
while the other endpoint $\mathbf w = \mathbf o + w_1 \mathbf f_1 + w_2 
\mathbf f_2$ satisfies
\begin{equation}
\label{eq:sf2}
w_1 > 0, \ w_2 < 0;
\end{equation}
\item
there exists a point on $Q$ having both positive coordinates in the frame 
$(\mathbf o;\mathbf{f}_1,\mathbf{f}_2)$.
\end{enumerate}
\end{definition}
\begin{remark}
Obviously, a frame can only split a slope if the slope has at least one edge.
\end{remark}

Suppose that a frame $(\mathbf o; \mathbf f_1, \mathbf f_2)$ splits a 
slope~$Q$ and let~$\mathbf z$ be the point where~$Q$ meets the ray 
$\{\mathbf{o}+\lambda \mathbf{f}_1 \colon \lambda\ge 0\}$.  If there is a 
supporting line for~$Q$ passing through~$\mathbf z$ that forms an angle $\le 
\pi/4$ with the ray, we say that the frame $(\mathbf o; \mathbf f_1, \mathbf 
f_2)$ \emph{forms small angle} with the slope~$Q$.

\begin{proposition}
\label{pr:sa}
Suppose that an integer frame $(\mathbf o; \mathbf f_1, \mathbf f_2)$ splits a 
slope~$Q$; then the frame $(\mathbf o; \mathbf f_2, \mathbf f_1)$ splits it as 
well, and at least one of the frames forms small angle with~$Q$.  If there 
exists a point $\mathbf{y} = \mathbf o + y_1 \mathbf f_1 + y_2 \mathbf f_2 \in 
Q$ such that $y_2 > 0$ and $y_1 + y_2 \le 0$, then $(\mathbf o; \mathbf f_1, 
\mathbf f_2)$ forms small angle with $Q$.
\end{proposition}

\begin{theorem}\label{th3-6}
Suppose that an integer frame $(\mathbf o; \mathbf f_1, \mathbf f_2)$ splits 
an integer slope $Q$ having $N$ edges and the endpoints $\mathbf v = \mathbf o 
+ v_1 \mathbf f_1 + v_2 \mathbf f_2$ and $\mathbf w = \mathbf o + w_1 \mathbf 
f_1 + w_2 \mathbf f_2$ satisfying~\eqref{eq:sf1} and~\eqref{eq:sf2}.  Then 
there exist $s \in \z$ and $t \in \z$ such that
\begin{gather}
0\le s\le t,
\label{eq:3-6A}
\\
v_2 - s \ge 0,
\label{eq:3-6B}
\\
- v_1 < ts-\frac{s^2-s}{2}+(v_2 - s)(t+1),
\label{eq:3-6C}
\\
2N\le v_2 + w_1 - t + s.
\label{eq:3-6D}
\end{gather}
Moreover, if $(\mathbf o; \mathbf f_1, \mathbf f_2)$ forms small angle with 
$Q$, we have
\begin{equation}
\label{eq:3-6E}
2N\le
v_2+w_1-t+s-\Ceil{\frac{-w_{2}}{2}} + 1
.
\end{equation}
\end{theorem}

\begin{corollary}
\label{cor3-6}
Under the hypotheses of Theorem~\ref{th3-6},
\begin{equation*}
2N\le v_2+w_1
,
\end{equation*}
and if $(\mathbf o; \mathbf f_1, \mathbf f_2)$ forms small angle with $Q$, 
then
\begin{equation*}
2N\le
v_2+w_1-\Ceil{\frac{-w_{2}}{2}} + 1
.
\end{equation*}
\end{corollary}

\begin{theorem}
\label{th3-8}
Under the hypotheses of Theorem~\ref{th3-6}, if the vertices of~$Q$ belong to 
a proper sublattice of $\zz$, then
\begin{equation*}
2N \le v_2 + w_1 - 1
.
\end{equation*}
\end{theorem}

There are four slopes naturally associated with a given convex polygon~$\m$.

Let $\m$ be an integer polygon in the plane.  Define
\begin{gather*}
\begin{array}{l}
\v=\max\{x_2:\;(x_1,x_2)\in\m\}, \\
\vl=\min\{x_1:\;(x_1,\v)\in\m\}, \\
\vp=\max\{x_1:\;(x_1,\v)\in\m\},
\end{array}
\begin{array}{l}
\n=\min\{x_2:\;(x_1,x_2)\in\m\}, \\
\nl=\min\{x_1:\;(x_1,\n)\in\m\}, \\
\np=\max\{x_2:\;(x_2,\n)\in\m\}, \\
\end{array}
\\
\begin{array}{l}
\l=\min\{x_1:\;(x_1,x_2)\in\m\}, \\
\ln=\min\{x_2:\;(\l,x_2)\in\m\}, \\
\lv=\max\{x_2:\;(\l,x_2)\in\m\}, \\
\end{array}
\begin{array}{l}
\p=\max\{x_1:\;(x_1,x_2)\in\m\}, \\
\pn=\min\{x_2:\;(\p,x_2)\in\m\}, \\
\pv=\max\{x_2:\;(\p,x_2)\in\m\}.
\end{array}
\end{gather*}
All these are integers.  Note that $(\nl,\n)$, $(\np,\n)$,
$(\vl,\v)$, $(\vp,\v)$, $(\l,\ln)$, $(\l,\lv)$, $(\p,\pn)$,
and $(\p,\pv)$ are (not necessarily distinct) vertices of $\m$.

Let us enumerate the vertices of~$\m$ starting from $\mathbf v_0 = (\l, \ln)$ 
and going in the positive direction until we come to $v_{N_4} = (\nl, \n)$.  
Clearly, the sequence $\mathbf v_0$, \dots, $\mathbf v_{N_4}$ gives rise to a 
slope with respect to the basis $(\mathbf e_1, \mathbf e_2)$.  We denote it 
by~$Q_4$.  Obviously,~$Q_4$ is an inclusion-wise maximal slope with respect to 
$(\mathbf e_1, \mathbf e_2)$ contained in the boundary of~$\m$.  Likewise, we 
define the slope $Q_1$ with respect to $(\mathbf e_2, - \mathbf e_1)$ having 
the endpoints $(\np, \n)$ and $(\p, \pn)$, the slope $Q_2$ with respect to $(- 
\mathbf e_1, - \mathbf e_2)$ having the endpoints $(\p, \pv)$ and $(\vp, \v)$, 
and the slope $Q_3$ with respect to $(- \mathbf e_2, \mathbf e_1)$ having the 
endpoints $(\vl, \v)$ and $(\l, \lv)$.  We call those \emph{maximal slopes} of 
the polygon~$P$ and denote by~$N_k$ the number of edges of~$Q_k$.

\begin{remark}
For each of the mentioned bases, the boundary of the polygon contains 
single-point maximal slopes apart from correspondent~$Q_k$.  However, we 
single~$Q_k$ out by explicitly indicating its endpoints.  For a given polygon, 
some of the maximal slopes~$Q_k$ may have but one vertex.
\end{remark}

Define
\begin{gather*}
M_1 =
\begin{cases}
0, & \text{if } \nl = \np,
\\
1, & \text{otherwise;}
\end{cases}
\ M_2 =
\begin{cases}
0, & \text{if } \pn = \pv,
\\
1, & \text{otherwise;}
\end{cases}
\\
M_3 =
\begin{cases}
0, & \text{if } \vl = \vp,
\\
1, & \text{otherwise;}
\end{cases}
\ M_4 =
\begin{cases}
0, & \text{if } \ln = \lv,
\\
1, & \text{otherwise.}
\end{cases}
\end{gather*}

\begin{proposition}
\label{pr:boundary}
Let~$\m$ be an $N$-gon; then each edge of~$\m$ either lies on a horizontal or 
a vertical line or it is the edge of exactly one of the maximal slopes 
of~$\m$; thus,
\begin{equation*}
N = \sum_{k = 1}^4 N_k + \sum_{k = 1}^4 M_k
.
\end{equation*}
\end{proposition}

\begin{proposition}
\label{pr:th5-3}
Let $\m$ be a convex integer polygon and $(\mathbf o; \mathbf f_1, \mathbf 
f_2)$ be an integer frame such that $\mathbf f_1, \mathbf f_2 \in \{\pm 
\mathbf e_1, \pm \mathbf e_2\}$.  Suppose that $\mathbf o$ does not belong 
to~$\m$ and the rays $\{\mathbf{o} +\lambda
\mathbf{f}_j \colon \lambda\ge 0\}$ ($j=1,2$) split $\m$; then $(\mathbf o; 
\mathbf f_1, \mathbf f_2)$ splits $Q_k$, where
$$k=\left\{\begin{array}{lllll}
1,&\text{if}&(\mathbf f_1, \mathbf f_2)=(-\mathbf{e}_1,\mathbf{e}_2)&\text{or}&(\mathbf f_1, \mathbf f_2)=(\mathbf{e}_2,-\mathbf{e}_1),\\
2,&\text{if}&(\mathbf f_1, \mathbf f_2)=(-\mathbf{e}_2,-\mathbf{e}_1)&\text{or}&(\mathbf f_1, \mathbf f_2)=(-\mathbf{e}_1,-\mathbf{e}_2),\\
3,&\text{if}&(\mathbf f_1, \mathbf f_2)=(\mathbf{e}_1,-\mathbf{e}_2)&\text{or}&(\mathbf f_1, \mathbf f_2)=(-\mathbf{e}_2,\mathbf{e}_1),\\
4,&\text{if}&(\mathbf f_1, \mathbf f_2)=(\mathbf{e}_2,\mathbf{e}_1)&\text{or}&(\mathbf f_1, \mathbf f_2)=(\mathbf{e}_1,\mathbf{e}_2).\\
\end{array}\right.$$
\end{proposition}

\begin{proposition}
\label{pr:th5-1}
Let $\m$ be a $\Gamma$-polygon, $S_1$ be the large $\mathbf{e}_1$-step of 
$\Gamma$, and $S_2$ be the large $\mathbf{e}_2$-step of $\Gamma$.  Then
\begin{gather*}
\np-\nl\ge S_1M_1,\ \pv-\pn\ge S_2M_2,\\
\vp-\vl\ge S_1M_3,\ \lv-\ln\ge S_2M_4.
\end{gather*}
\end{proposition}

\section{Type III polygons}
\label{sec:iii}

In this section we prove Theorem~\ref{subtheorem-c} for type~III polygons.

\begin{lemma}
\label{lem8-1}
Given an integer $n \ge 2$ and a type~III$_n$ polygon $\m$, the following 
assertions hold:
\begin{enumerate}[(i)]
\item
The frame $((n,0);\mathbf{e}_2,-\mathbf{e}_1)$ splits $Q_1$.
\item
The frame $((n,n);-\mathbf{e}_2,-\mathbf{e}_1)$ splits $Q_2$.
\item
The following inequalities hold:
\begin{gather}
\l\ge0,
\label{eq1-lem8-1}
\\
\vp\le n-1,
\label{eq2-lem8-1}
\\
\np\le n-1,
\label{eq3-lem8-1}
\\
\p\ge n+1,
\label{eq4-lem8-1}
\\
\n<\pn<0.
\label{eq5-lem8-1}
\end{gather}
\item
The intersection of $\m$ with the open half-plane $x_1<n$ is contained in the 
slab $0\le x_1<n$.   All the vertices of $\m$ belonging to the closed 
half-plane $x_1\ge n$ lie on the lines
\begin{equation}\label{eqpr-lem8-1}
x_2=k\qquad(k=1,\ldots,n-1),
\end{equation}
each line containing at most one vertex.
\item
If the vertices of~$\m$ belong to a $(1,n)$-lattice~$\Gamma$, then the large
$\mathbf{e}_2$- and $\mathbf{e}_1$-steps of $\Gamma$ are greater then or equal 
to $2$.
\end{enumerate}
\end{lemma}
\begin{proof}
Assertions (i) and (ii) follow from the definition of a type~III$_n$ polygon 
and Proposition~\ref{pr:th5-3}.  Assertions (iii) and (iv) are obvious.  
According to Proposition~\ref{pr:th1-11}, to prove (v), it suffices to show 
that the small $\mathbf e_1$- and $\mathbf e_2$-steps of $\Gamma$ are less 
then $n$.  As the vertex $(\p,\pn)$ belongs to $\Gamma$ and 
satisfies~\eqref{eq5-lem8-1}, we see that indeed the small $\mathbf{e}_2$-step 
of~$\Gamma$ is less than~$n$.  Further, it follow from~\eqref{eq2-lem8-1}, 
\eqref{eq3-lem8-1}, and~\eqref{eq5-lem8-1} that the vertices $(\vp,\v)$ and 
$(\np,\n)$ lie in the slab $0\le x_1<n$.  Suppose, contrary to our claim, that 
the small $\mathbf{e}_1$-step of $\Gamma$ equals $n$.  Then we see that 
$\vp=\np=0$ and consequently, $P$ contains the segment $[(0,\n),(0,\v)]$.  
However, we obviusly have $\n < 0$ and $\v > n$, so the segment contains the 
points $\mathbf{0}, (0,n)\in \nz$, which is impossible, as~$\m$ is free of 
$\nz$-points.
\end{proof}

\begin{lemma}
\label{lem8-2}
Suppose that $n \ge 3$.  Let $\Gamma$ and $b$ be a lattice and a number such 
that either $\Gamma=\zz$ and $b = 0$ or $\Gamma$ is a $(1,n/2)$-lattice with 
the basis $(\mathbf e_1+a\mathbf e_2, (n/2)\mathbf e_2)$, where the integer $a$ 
satisfies $1\le a\le n/2-1$, and $b = 1$.  Let $\m$ be a type~III$_n$ $N$-gon 
with the vertices belonging to $\Gamma$.  Then
\begin{equation}\label{eq1-lem8-2}
N\le2n+2-2b.
\end{equation}
\end{lemma}
\begin{proof}
We begin by translating the geometrical constraints on $P$ into inequalities.

The frame $((n,0);\mathbf{e}_2,-\mathbf{e}_1)$ splits $Q_1$, so by 
Theorem~\ref{th3-6} there exist integers $s_1$ and $t_1$ such that
\begin{gather}
2N_1\le\pn-\np+n-t_1+s_1,
\label{eq2-lem8-2}
\\
-\np+n-s_1\ge 0,
\label{eq3-lem8-2}
\\
-\n<t_1s_1-\frac{s_1^2-s_1}{2}+(-\np+n-s_1)(t_1+1),
\label{eq4-lem8-2}
\\
0\le s_1\le t_1.
\label{eq5-lem8-2}
\end{gather}
Likewise, $((n,n);-\mathbf{e}_2,-\mathbf{e}_1)$ splits $Q_2$, so there exist 
integers~$s_2$ and~$t_2$ such that
\begin{gather}
2N_2\le-\pv-\vp+2n-t_2+s_2,
\label{eq6-lem8-2}
\\
-\vp+n-s_2\ge 0,
\label{eq7-lem8-2}
\\
\v-n<t_2s_2-\frac{s_2^2-s_2}{2}+(-\vp+n-s_2)(t_2+1),
\label{eq8-lem8-2}
\\
0\le s_2\le t_2.
\label{eq9-lem8-2}
\end{gather}

As~$Q_3$ is a slope with respect to~$(\mathbf{e}_1,-\mathbf{e}_2)$, by 
Proposition~\ref{pr:slp} there exists $s_3 \in \z$ such that
\begin{gather}
2N_3\le\vl-\l+s_3,
\label{eq10-lem8-2}
\\
\v-\lv\ge\frac12s_3(s_3+1),
\label{eq11-lem8-2}
\\
0\le s_3\le N_3.
\label{eq12-lem8-2}
\end{gather}

Likewise, applying Proposition~\ref{pr:slp} to~$Q_4$ and 
$(\mathbf{e}_1,\mathbf{e}_2)$, we conclude that there exists $s_4 \in \z$ such 
that
\begin{gather}
2N_4\le\nl-\l+s_4,
\label{eq13-lem8-2}
\\
\ln-\n\ge\frac12s_4(s_4+1),
\label{eq14-lem8-2}
\\
0\le s_4\le N_4.
\label{eq15-lem8-2}
\end{gather}
Further, by Proposition~\ref{pr:th5-1},
\begin{gather}
\np-\nl\ge(1+b)M_1,
\label{eq16-lem8-2}
\\
\pv-\pn\ge(1+b)M_2,
\label{eq17-lem8-2}
\\
\vp-\vl\ge(1+b)M_3,
\label{eq18-lem8-2}
\\
\lv-\ln\ge(1+b)M_4,
\label{eq19-lem8-2}
\\
\l\ge bM_4.
\label{eq20-lem8-2}
\end{gather}
Indeed, if $b = 0$, \eqref{eq16-lem8-2}--\eqref{eq19-lem8-2} immediately 
follow from the proposition.  Suppose that $b = 1$; then the large $\mathbf 
e_2$-step of $\Gamma$ is $n/2 \ge 2$, and as $a \ge 1$, by virtue 
of~Proposition~\ref{pr:th1-12} we have $\mathbf{e}_1\notin\Gamma$, and 
consequently, the large $\mathbf{e}_2$-step of $\Gamma$ is~$\ge 2$ as well.

By Lemma~\ref{lem8-1}, we have $\l \ge 0$, so to prove \eqref{eq20-lem8-2} it 
suffices to show that $M_4=0$ provided that $b=1$ and $\l=0$.  Indeed, in this 
case $\Gamma$ has large $\mathbf{e}_2$-step $n/2$, so every other point of 
$\Gamma$ lying on the line $x_1=0$ belongs to~$\nz$.  Therefore, $\m$ cannot 
have two vertices on this line and $M_4=0$ if $\l=0$.

To prove the lemma, we argue by contradiction, assuming that
\begin{equation}\label{eq1'-lem8-2}
2N\ge 4n+6-4b.
\end{equation}

Summing~\eqref{eq2-lem8-2}, \eqref{eq6-lem8-2}, \eqref{eq10-lem8-2}, and 
\eqref{eq13-lem8-2} and subsequently using 
\eqref{eq16-lem8-2}--\eqref{eq20-lem8-2}, we obtain:
\begin{multline*}
2N=\sum_{k=1}^42N_k+\sum_{k=1}^42M_k
\\
\le 3n+s_1+s_2+s_3+s_4-t_1-t_2
\\
-(\np-\nl)-(\pv-\pn)-(\vp-\vl)
\\-2\l+2M_1
+2M_2+2M_3+2M_4
\\
\le 3n+s_1+s_2+s_3+s_4-t_1-t_2
\\
+(1-b)M_1+(1-b)M_2+(1-b)M_3+(2-2b)M_4.
\end{multline*}
Comparing this with~(\ref{eq1'-lem8-2}), we deduce
\begin{multline}
\label{eq21-lem8-2}
n-s_1-s_2-s_3-s_4+t_1+t_2
\\
-(1-b)M_1-(1-b)M_2-(1-b)M_3-(2-2b)M_4+6-4b\le 0
.
\end{multline}

Now we use~\eqref{eq4-lem8-2} and~\eqref{eq8-lem8-2} to estimate $\v-\n$ from 
above:
\begin{multline}
\label{eq22-lem8-2}
\v-\n<n+t_1s_1-\frac{s_1^2-s_1}{2}+t_2s_2-\frac{s_2^2-s_2}{2}
\\
+(-\np+n-s_1)(t_1+1)
+(-\vp+n-s_2)(t_2+1).
\end{multline}

Let us estimate~$\np$ and~$\vp$.  Using \eqref{eq16-lem8-2}, 
\eqref{eq13-lem8-2}, \eqref{eq15-lem8-2}, and \eqref{eq1-lem8-1}, we obtain
\begin{multline*}
\np\ge\nl+(1+b)M_1\ge2N_4+\l-s_4+(1+b)M_1
\\
\ge s_4+\l+(1+b)M_1
\ge s_4+M_1,
\end{multline*}
whence
\begin{equation}\label{eq23-lem8-2}
-\np+n-s_1\le n-s_1-s_4-M_1.
\end{equation}
Incidentally, note that the left-hand side is nonnegative by virtue 
of~\eqref{eq3-lem8-2}, so
\begin{equation}\label{eq24-lem8-2}
n-s_1-s_4-M_1\ge0.
\end{equation}
Likewise, from \eqref{eq18-lem8-2}, \eqref{eq11-lem8-2}, \eqref{eq12-lem8-2}, 
and \eqref{eq1-lem8-1} we derive
\begin{equation}
\label{eq25-lem8-2}
-\vp+n-s_2\le n-s_2-s_3-M_3,
\end{equation}
which together with~\eqref{eq7-lem8-2}
implies
\begin{equation}
\label{eq26-lem8-2}
n-s_2-s_3-M_3\ge0.
\end{equation}

As $t_1+1>0$ and $t_2+1>0$, we can use \eqref{eq23-lem8-2} and 
\eqref{eq25-lem8-2} to obtain from~\eqref{eq22-lem8-2}
\begin{multline}
\label{eq27-lem8-2}
\v-\n<n+t_1s_1-\frac{s_1^2-s_1}{2}+t_2s_2-\frac{s_2^2-s_2}{2}
\\
+(n-s_1-s_4-M_1)(t_1+1)+(n-s_2-s_3-M_3)(t_2+1).
\end{multline}

Now we estimate $\v-\n$ from below by summing \eqref{eq11-lem8-2}, 
\eqref{eq14-lem8-2}, and \eqref{eq19-lem8-2}:
\begin{equation}
\label{eq28-lem8-2}
\v-\n\ge(1+b)M_4+\frac12s_3(s_3+1)+\frac12s_4(s_4+1).
\end{equation}
Consider the second term on the right-hand side.  
Inequality~\eqref{eq21-lem8-2} gives
\begin{multline*}
s_3-1\ge(n-s_1-s_4+t_1-M_1)+(t_2-s_2)
\\
+(5-4b+bM_1-(1-b)M_2-(1-b)M_3-(2-2b)M_4).
\end{multline*}
The second term on the right-hand side is nonnegative by virtue 
of~\eqref{eq9-lem8-2} and the third one is also nonnegative (even positive), 
which is easily seen by separately checking $b=0$ and $b = 1$.  Consequently, 
we have
\begin{equation}
\label{eq29-lem8-2}
s_3-1\ge n-s_1-s_4+t_1-M_1.
\end{equation}
By virtue of~\eqref{eq24-lem8-2} we have $n-s_1-s_4+t_1-M_1\ge t_1\ge0$, so 
using~\eqref{eq29-lem8-2}, we get
\begin{multline*}
\frac12s_3(s_3+1)=
s_3 + \frac12s_3(s_3-1)
\\
\ge s_3 + \frac12(n-s_1-s_4+t_1-M_1+1)(n-s_1-s_4+t_1-M_1).
\end{multline*}
Set
\begin{equation*}
A=n-s_1-s_4-M_1, \quad B=t_1+1
\end{equation*}
($A$ and $B$ are integers) and continue as follows:
\begin{multline*}
\frac12s_3(s_3+1)\ge s_3 + \frac12(A+B)(A+B-1)
\\
=s_3 + \frac12(A^2-A)+ \frac 12 (B^2-B)+AB
\ge s_3 + \frac12(B^2-B)+AB.
\end{multline*}
For the terms on the right-hand side we have
\begin{multline*}
\frac12
(B^2-B)
= \frac12 (t_1^2+t_1)
=t_1s_1-\frac{s_1^2-s_1}{2}+\frac12(t_1-s_1)(t_1-s_1+1)
\\
\ge t_1s_1-\frac{s_1^2-s_1}{2}
\end{multline*}
(since $t_1-s_1\ge0$ according to~\eqref{eq5-lem8-2}), and
$$AB=(n-s_1-s_4-M_1)(t_1+1),$$
and we finally obtain
\begin{equation}
\label{eq30-lem8-2}
\frac12 s_3(s_3+1)\ge
s_3 + t_1s_1-\frac{s_1^2-s_1}{2}+(n-s_1-s_4-M_1)(t_1+1).
\end{equation}

One can estimate the third term on the right-hand side of \eqref{eq28-lem8-2} 
in much the same way by making use of~\eqref{eq5-lem8-2},  
\eqref{eq26-lem8-2}, and~\eqref{eq9-lem8-2}.  Eventually,
\begin{equation}
\label{eq31-lem8-2}
\frac12 s_4(s_4+1)\ge
s_4 + t_2s_2-\frac{s_2^2-s_2}{2}+(n-s_2-s_3-M_3)(t_2+1).
\end{equation}

Now, using~\eqref{eq30-lem8-2} and~\eqref{eq31-lem8-2}, we derive 
from~\eqref{eq28-lem8-2} the following estimate:
\begin{multline}
\label{eq32-lem8-2}
\v-\n\ge(1+b)M_4+s_3+s_4+t_1s_1-\frac{s_1^2-s_1}{2}+t_2s_2-\frac{s_2^2-s_2}{2}
\\
+(n-s_1-s_4-M_1)(t_1+1)+(n-s_2-s_3-M_3)(t_2+1).
\end{multline}

Comparing~\eqref{eq27-lem8-2} with~\eqref{eq32-lem8-2}, we obtain
$$-n+s_3+s_4+(1+b)M_4<0.$$
Summing this inequality with~\eqref{eq21-lem8-2}, we get
$$(t_1-s_1)+(t_2-s_2)+(6-4b-(1-b)M_1-(1-b)M_2-(1-b)M_3-(1-3b)M_4)<0.$$
However, the summands on the left-hand side are nonnegative.  Indeed, in the 
case of the first and the second ones it follows from~\eqref{eq5-lem8-2} 
and~\eqref{eq9-lem8-2}, respectively.  In the case of the third summand for $b 
= 0$ we have
\begin{multline*}
6-4b-(1-b)M_1-(1-b)M_2-(1-b)M_3-(1-3b)M_4
\\
=6-M_1-M_2-M_3-M_4\ge2,
\end{multline*}
while for $b=1$ we have
\begin{equation*}
6-4b-(1-b)M_1-(1-b)M_2-(1-b)M_3-(1-3b)M_4=2+2M_4\ge2.
\end{equation*}
This contradiction proves the lemma.
\end{proof}

\begin{lemma}
\label{lem8-3}
Suppose that $n \ge 4$ is even and $\m$ is a type~III$_n$ $N$-gon with 
vertices belonging to a $(1,n/2)$-lattice $\Gamma$ having the basis 
$(\mathbf{e}_1, (n/2)\mathbf{e}_2)$; then
$$N\le2n.$$
\end{lemma}
\begin{proof}
According to Lemma~\ref{lem8-1}, all the vertices of~$\m$ belonging to the 
open half-plane $x_1<n$ lie in the slab $0\le x_1<n$.  All the points 
of~$\Gamma$ belonging to this slab lie on the lines $x_1=i$ 
$(i=0,1,\ldots,n-1)$, and each of these lines contains at most two vertices 
except for $x_1 = 0$, which contain at most one (since every other point of 
$\Gamma$ lying on this line belongs to $\nz$).   This gives the maximum of 
$2n-1$ lying in the said half-plane.

It remains to prove that at most one vertex lies in the half-plane~$x_1\ge n$.  
According to Lemma~\ref{lem8-1}, each of the lines~\eqref{eqpr-lem8-1} 
contains at most one vertex, and there are no other vertices.  But among these 
only the line~$x_2=n/2$ has points belonging to~$\Gamma$.
\end{proof}

\begin{lemma}
\label{lem8-4}
Suppose that $n \ge 3$ and $\m$ is a type~III$_n$ $N$-gon with vertices 
belonging to a $(1,n)$-lattice $\Gamma$ having the basis  
$(\mathbf{e}_1+a\mathbf{e}_2,n\mathbf{e}_2)$, where $1\le a\le n-1$.  
Then
\begin{equation}\label{eq1-lem8-4}
N\le 2n-2.
\end{equation}
\end{lemma}
\begin{proof}
The frame $((n,0),\mathbf{e}_2,-\mathbf{e}_1)$ splits $Q_1$, and the vertices 
of~$Q_1$ belong to a proper subset of $\zz$, so by Theorem~\ref{th3-8} we have
\begin{equation}
\label{eq3-lem8-4}
2N_1\le\pn-\np+n-1.
\end{equation}
Applying the same theorem to $((n,n),-\mathbf{e}_2,-\mathbf{e}_1)$ and $Q_2$, 
we obtain
\begin{equation}
\label{eq4-lem8-4}
2N_2\le-\pv-\vp+2n-1.
\end{equation}

As $Q_3$ is a slope with respect to $(\mathbf{e}_1,-\mathbf{e}_2)$ and 
$(\mathbf{e}_1-a(-\mathbf{e}_2), n(-\mathbf{e}_2))$ is a basis of $\Gamma$ 
(Proposition~\ref{pr:th1-1}), we evoke Proposition~\ref{pr:slp} and conclude 
that there exists an integer~$s_3$ such that
\begin{gather}
2N_3\le\vl-\l+s_3,
\label{eq5-lem8-4}
\\
\v-\lv\ge\frac{2a+(s_3-1)n}{2}s_3,
\label{eq6-lem8-4}
\\
0\le s_3\le N_3.
\label{eq7-lem8-4}
\end{gather}
Likewise, applying Proposition~\ref{pr:slp} to $Q_4$, the basis 
$(\mathbf{e}_1,\mathbf{e}_2)$ and the basis $(\mathbf{e}_1-(n-a)\mathbf{e}_2, 
n\mathbf{e}_2)$ of~$\Gamma$, we obtain an integer $s_4$ such that
\begin{gather}
2N_4\le\nl-\l+s_4,
\label{eq8-lem8-4}
\\
\ln-\n\ge\frac{2(n-a)+(s_4-1)n}{2}s_4,
\label{eq9-lem8-4}
\\
0\le s_4\le N_4.
\label{eq10-lem8-4}
\end{gather}

Finally, we have
\begin{gather}
\np-\nl\ge 2M_1,
\label{eq11-lem8-4}
\\
\pv-\pn\ge nM_2,
\label{eq12-lem8-4}
\\
\vp-\vl\ge2M_3,
\label{eq13-lem8-4}
\\
\lv-\ln\ge nM_4,
\label{eq14-lem8-4}
\\
\l\ge1.
\label{eq15-lem8-4}
\end{gather}
Indeed, the large $\mathbf e_2$-step of $\Gamma$ is $n$, so by 
Proposition~\ref{pr:th5-1} we have \eqref{eq12-lem8-4} 
and~\eqref{eq14-lem8-4}.  Because $1\le a\le n-1$, we have 
$\mathbf{e}_1\notin\Gamma$ and the large $\mathbf{e}_1$-step of $\Gamma$ is 
greater then or equal to~2, so inequalities \eqref{eq11-lem8-4} and 
\eqref{eq13-lem8-4} hold by virtue of the same proposition.  Finally, 
\eqref{eq15-lem8-4} follows from the fact that $\l$ is nonnegative by 
Lemma~\ref{lem8-1}, and the fact that all the points of~$\Gamma$ lying on the 
line~$x_1=0$ belong to~$\nz$.

Assuming that~\eqref{eq1-lem8-4} does not hold, we have
\begin{equation}\label{eq1'-lem8-4}
2N\ge4n-2.
\end{equation}
Estimate $2N$ from above by summing \eqref{eq3-lem8-4},
\eqref{eq4-lem8-4}, \eqref{eq5-lem8-4}, and \eqref{eq8-lem8-4} and 
subsequently using \eqref{eq11-lem8-4}--\eqref{eq13-lem8-4} and 
\eqref{eq15-lem8-4}:
\begin{multline*}
2N=\sum_{k=1}^42N_k+\sum_{k=1}^42M_k
\\
\le 3n-2+s_3+s_4+(2M_1-(\np-\nl))
+(nM_2-(\pv-\pn))
\\
+(2M_3-(\vp-\vl))
-(n-2)M_2+2(M_4-\l)
\\
\le3n-2+s_3+s_4.
\end{multline*}
Comparing this with~\eqref{eq1'-lem8-4}, we get
\begin{equation}
\label{eq16-lem8-4}
s_3+s_4\ge n.
\end{equation}

Summing \eqref{eq6-lem8-4}, \eqref{eq9-lem8-4}, and \eqref{eq14-lem8-4} and 
discarding the nonnegative term~$nM_4$, we obtain a lower estimate of $\v-\n$:
\begin{equation}\label{eq17-lem8-4}
\v-\n\ge\frac{2a+(s_3-1)n}{2}s_3+\frac{2(n-a)+(s_4-1)n}{2}s_4.
\end{equation}

Now, a simple geometrical reasoning provides the upper estimate
\begin{equation}\label{eq19-lem8-4}
\v-\n\le n^2-1.
\end{equation}
Indeed, note that the triangle with the vertices $(\vp, \v)$, $(\np, \n)$, and 
$(\p, \pn)$ is contained in $\m$, so the segment being the intersection of the 
triangle with the line $x_1 = n$ lies between two adjacent points of $\nz$.  
As the distance from $(\p, \pn)$ to the line is greater than or equal to 1, it 
is not hard to see that the projection of the segment $[(\vp, \v), (\np, \n)]$ 
onto $x_1 = 0$ has length strictly less than $n^2$, whence~\eqref{eq19-lem8-4} 
follows.

Comparing~\eqref{eq17-lem8-4} and~\eqref{eq19-lem8-4}, we 
obtain
\begin{equation}
\label{eq20-lem8-4}
\frac{2a+(s_3-1)n}{2}s_3+\frac{2(n-a)+(s_4-1)n}{2}s_4\le
n^2-1.
\end{equation}
Let us estimate the second term on the left-hand side.  It follows 
from~\eqref{eq5-lem8-4} and~\eqref{eq7-lem8-4} that
$$2s_3\le\vl-\l+s_3,$$
whence using~\eqref{eq15-lem8-4} and~\eqref{eq2-lem8-1} we obtain
$$s_3\le\vl-\l\le\vp-1\le n-2.$$
Combining this with~\eqref{eq16-lem8-4} we deduce
$$0<n-s_3\le s_4.$$
Consequently, we have
$$\frac{2(n-a)+(s_4-1)n}{2}s_4\ge\frac{2(n-a)+(n-s_3-1)n}{2}(n-s_3),$$
and from~\eqref{eq20-lem8-4} we obtain
$$\frac{2a+(s_3-1)n}{2}s_3+\frac{2(n-a)+(n-s_3-1)n}{2}(n-s_3)-n^2+1\le0.$$
Transforming the left-hand side, we can write the inequality in the form
$$n\left(s_3+\frac{2a-n-n^2}{2n}\right)^2+\frac{a(n-a)}{n}+\frac14(n+1)(n-1)(n-4)\le0.$$
Clearly, this inequality cannot hold with $n\ge4$.  This contradiction proves 
the lemma in the case $n \ge 4$.

If $n = 3$, it follows from~\eqref{eq2-lem8-1} and~\eqref{eq3-lem8-1} that
\begin{equation}\label{eq21-lem8-4}
\vp+\np\le4,
\end{equation}
inequalities~\eqref{eq5-lem8-4}, \eqref{eq7-lem8-4}, and \eqref{eq15-lem8-4} 
give
$$\vl\ge s_3+\l+(2N_3-2s_3)\ge s_3+1,$$
and inequalities \eqref{eq8-lem8-4}, \eqref{eq10-lem8-4}, 
and~\eqref{eq15-lem8-4} similarly imply
$$\nl\ge s_4+1.$$
From these estimates and~\eqref{eq16-lem8-4} we derive
$$\vp+\np\ge\vl+\nl\ge s_3+s_4+2\ge n+2=5,$$
which contradicts~\eqref{eq21-lem8-4}.
\end{proof}

\begin{lemma}
\label{lem8-5}
Suppose that $n \ge 3$.  Let $\Gamma$ and $b$ be a lattice and a number such 
that either $\Gamma$ has invariant factors $(1,n)$ and $b = 0$ or $\Gamma$ has 
invariant factors $(1,n/2)$, and $b = 1$.  Suppose that the small $\mathbf 
e_1$-step of $\Gamma$ is greater then $1$, and let $\m$ be a type~III$_n$ 
$N$-gon with the vertices belonging to $\Gamma$.  Then
\begin{equation}\label{eq0-lem8-5}
N\le 2n-2+2b.
\end{equation}
\end{lemma}
\begin{proof}
As in the proof of Lemma~\ref{lem8-4}, we have~\eqref{eq3-lem8-4} 
and~\eqref{eq4-lem8-4}.  Applying Proposition~\ref{pr:slp} to~$Q_3$ and~$Q_4$, 
we obtain
\begin{gather}
2N_3\le\vl-\l,
\label{eq3-lem8-5}
\\
2N_4\le\nl-\l.
\label{eq4-lem8-5}
\end{gather}
As the small $\mathbf e_1$-step of $\Gamma$ is greater then 1, we still 
have~\eqref{eq11-lem8-4} and~\eqref{eq13-lem8-4}.  By 
Proposition~\ref{pr:th5-1}, we also have
\begin{equation}
\label{eq7-lem8-5}
\pv-\pn\ge(2-b)M_2
\end{equation}
(in the case $b = 0$ this follows from assertion~(v) of Lemma~\ref{lem8-1}). 

Summing~\eqref{eq3-lem8-4}, \eqref{eq4-lem8-4}, \eqref{eq3-lem8-5}, and 
\eqref{eq4-lem8-5} and subsequently applying~\eqref{eq11-lem8-4}, 
\eqref{eq13-lem8-4}, \eqref{eq7-lem8-5}, and~\eqref{eq1-lem8-1}, we obtain the 
estimate
\begin{multline*}
2N=\sum_{k=1}^42N_k+\sum_{k=1}^42M_k
\\
\le
3n+(2M_1-(\np-\nl))+((2-b)M_2-(\pv-\pn))
\\
+(2M_3
-(\vp-\vl))+2(M_4-1)+bM_2
\le3n+b,
\end{multline*}
whence
$$N\le\frac32n+\frac{b}{2}=(2n-2+2b)+\frac{-n-3b+4}{2}.$$
Since $n\ge3$ and $b\ge0$, we can write
$$N\le(2n-2+2b)+\frac12,$$
which yields \eqref{eq0-lem8-5}.
\end{proof}

\begin{proof}[Proof of Theorem~\ref{subtheorem-c} for type~III polygons]
Let $\m$ be a type~III$_n$ $N$-gon.  By Lemma~\ref{lem8-2}, its number of 
vertices satisfies $N \le 2n + 2$.

Suppose that the vertices of~$\m$ belong to a $(1, n)$-lattice~$\Gamma$.  If 
the small $\mathbf e_1$-step of~$\Gamma$ is greater then~1, by 
Lemma~\ref{lem8-5} we have $N \le 2n - 2$.  If the small $\mathbf e_1$-step 
of~$\Gamma$ equals~1, by Proposition~\ref{pr:th1-12} this lattice admits a 
basis of the form $(\mathbf e_1 + a \mathbf e_2, n\mathbf e_2)$, where $0 \le 
a \le n-1$. According to assertion~(v) of Lemma~\ref{lem8-1}, we cannot have 
$a = 0$, so Lemma~\ref{lem8-4} provides the same bound on the number of 
vertices.

Finally, suppose that~$n$ is even and the vertices of~$\m$ belong to a $(1, 
n/2)$-lattice~$\Gamma$.  If the small $\mathbf e_1$-step of~$\Gamma$ is 
greater than~1, by Lemma~\ref{lem8-5} we have $N \le 2n$.  Otherwise, $\Gamma$ 
has a basis of the form $(\mathbf e_1 + a\mathbf e_2, (n/2)\mathbf e_2)$, 
where $0 \le a \le n/2 - 1$; then Lemma~\ref{lem8-2} gives the same estimate 
in case $a \ne 0$ and Lemma~\ref{lem8-3}, in case $a = 0$.
\end{proof}

\section{Type IV polygons}
\label{sec:iv}

In this section we prove Theorem~\ref{subtheorem-c} for type~IV polygons.

Throughout the section, we fix an integer $n \ge 3$.

\begin{lemma}\label{lem9-1}
Suppose that the line $x_1-x_2=n$ splits a type~IV$_n$ polygon; then so does 
one of the segments $[(0,-n),(n,0)]$ and $[(n,0),(2n,n)]$.
\end{lemma}
\begin{proof}
All the points of the line $x_1 - x_2 = n$ belonging to the slab $-n+1\le 
x_1\le 2n-1$ lie on the segments $[(-n,-2n),(0,-n)]$, $[(0,-n),(n,0)]$, and $
[(n,0),(2n,n)]$, and as the polygon is free of $\nz$-points, exactly one of 
the segments splits it.  However, it cannot be the first one, because it 
follows from Corollary~\ref{cor:Helly} that the polygon has no points with 
both nonpositive coordinates.
\end{proof}
\begin{figure}
\label{fig:type4a}
\includegraphics{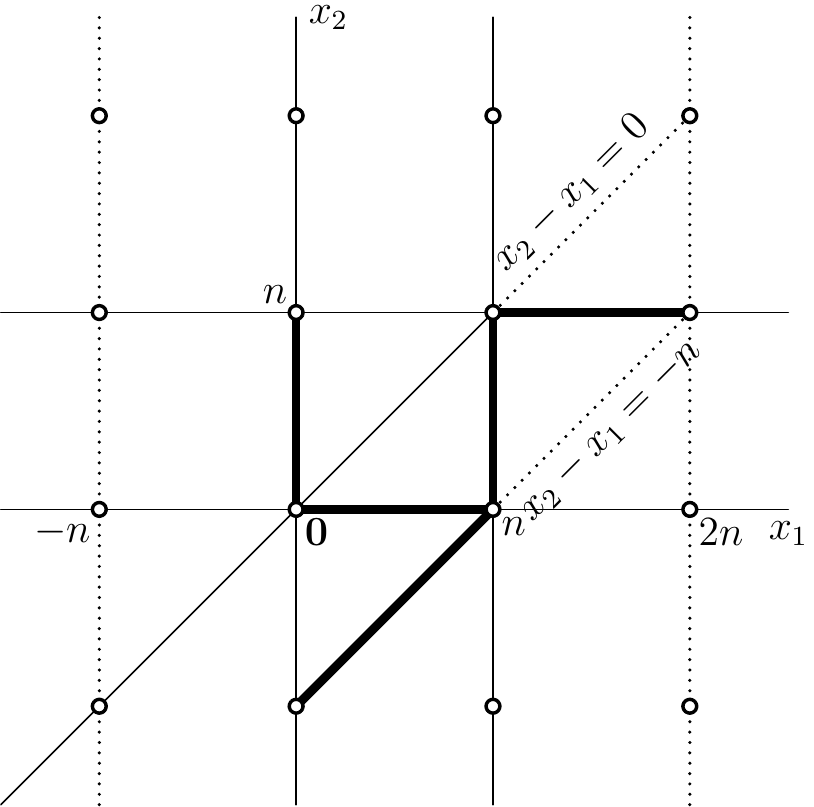}
\caption{The segments splitting the polygon in the hypothesis of 
Lemma~\ref{lem9-2} are thick, and $\m$ does not intersect dotted lines.  The 
inequalities \eqref{eq4-lem9-2}--\eqref{eq7-lem9-2} are obvious.}
\end{figure}
\begin{lemma}
\label{lem9-2}
Suppose that $\m$ is a type~IV$_n$ polygon and the segment $[(0,-n),(n,0)]$ 
splits it.  Then the following assertions hold:
\begin{enumerate}[(i)]
\item
The intersection of $\m$ with the half-plane $x_1 \ge n$ lies in the slab
\begin{equation}\label{eq1-lem9-2}
-n<x_2 - x_1<0
.
\end{equation}
\item
The frame $((n,0);\mathbf{e}_2,-\mathbf{e}_1)$ splits the slope $Q_1$ and 
forms small angle with it.
\item
The frame $((n,n);\mathbf{e}_1,-\mathbf{e}_2)$ splits the slope $Q_3$ and 
forms small angle with it.
\item
The frame $(\mathbf{0};\mathbf{e}_1,\mathbf{e}_2)$ splits the slope $Q_4$.
\item
The slope $Q_1$ has a vertex $\mathbf{v} = (v_1, v_2)$ 
satisfying
\begin{equation}
\label{eq3-lem9-2}
v_2 - v_1 \le -n-1
.
\end{equation}
\item
The following inequalities hold:
\begin{gather}
n<\vp\le\p\le 2n-1,
\label{eq4-lem9-2}
\\
n<\v\le\vp-1,
\label{eq5-lem9-2}
\\
-n<\l<0,
\label{eq6-lem9-2}
\\
0<\lv<n.
\label{eq7-lem9-2}
\end{gather}
\end{enumerate}
\end{lemma}
\begin{proof}
To prove (i), it suffices to observe that $\m$ cannot have common points with 
the segments $[(n,0), (2n, n)]$ and $[(n, n), (2n, 2n)]$.  Indeed, if $\m$ had 
common points with the former segment, by convexity it would containt the 
point $(n, 0)$; if it had common points with the latter, it would contain the 
point $(n, n)$ by Lemma~\ref{Helly} applied to $\m$ and the lines $x_2 - x_1 = 
0$ and $x_2 = n$.  Thus, the part of $\m$ contained in the half-plane $x_1 \ge 
n$ must lie between the lines $x_2 - x_1 = -n$ and $x_2 - x_1 = 0$.

Let us prove (v).  Clearly, the functional $x_2 - x_1$ attains its maximum 
on~$\m$ on a vertex~$\mathbf v \in Q_1$.  As the line $x_2 - x_1 = -n$ 
splits~$\m$, this minimum is less than~$n$, and~(v) follows.

The fact that the frames split correspondent slopes in assertions (ii)--(iv) 
follows from Proposition~\ref{pr:th5-3}.  To prove that 
$((n,0);\mathbf{e}_2,\mathbf{e}_1)$ forms small angle with $Q_1$, we apply 
Proposition~\ref{pr:sa} taking the vertex from assertion (v) as $\mathbf y$.  
To prove that $((n,n);\mathbf{e}_1,-\mathbf{e}_2)$ splits $Q_3$, we use the 
same theorem with $\mathbf y = (\l, \lv)$.

The inequalities in (vi) are is fairly intuitive, see Figure~\ref{fig:type4a}.
\end{proof}

\begin{lemma}
\label{lem9-3}
Let $\Gamma$ and $b$ be a sublattice of $\zz$ and a number such that either 
$\Gamma = \zz$ and $b = 0$ or $\Gamma$ has a basis of the form 
$(\mathbf{e}_1+a\mathbf{e}_2$, $(n/2)\mathbf{e}_2)$, where $1\le
a\le n/2-1$, and $b = 1$ (this is only possible if $n$ is even).  Let $\m$ be 
a type~IV$_n$ $N$-gon with vertices belonging to $\Gamma$, and suppose that 
the segment $[(0,-n),(n,0)]$ splits $\m$.  Then
\begin{equation}\label{eq1-lem9-3}
N\le 2n+2-2b.
\end{equation}
\end{lemma}
\begin{proof}
By Lemma~\ref{lem9-2}, the frame $((n,0);\mathbf{e}_2,-\mathbf{e}_1)$ forms 
small angle with the slope~$Q_1$, so by Corollary~\ref{cor3-6} we have
\begin{equation}
\label{eq2-lem9-3}
2N_1\le\pn-\np+n-
\Ceil{\frac{\p-n}{2}} + 1.
\end{equation}

Likewise, as $((n,n);\mathbf{e}_1,-\mathbf{e}_2)$ forms small angle 
with~$Q_3$, we obtain
\begin{equation}\label{eq6-lem9-3}
2N_3\le\vl-\lv-
\Ceil{\frac{\v-n}{2}}
+ 1
.
\end{equation}

Applying Proposition~\ref{pr:slp} to the basis $(-\mathbf{e}_1,-\mathbf{e}_2)$ 
and the slope $Q_2$, we see that there exists an integer $s_2$ such that
\begin{gather}
2N_2\le\p-\vp+s_2,
\label{eq3-lem9-3}
\\
\v-\pv\ge\frac{s_2^2+s_2}{2},
\label{eq4-lem9-3}
\\
0\le s_2\le N_2.
\label{eq5-lem9-3}
\end{gather}

As the frame $(\mathbf{0};\mathbf{e}_1,\mathbf{e}_2)$ splits $Q_4$, by 
Corollary~\ref{cor3-6} and Theorem~\ref{th3-8} we have
\begin{equation}
\label{eq7-lem9-3}
2N_4\le\nl+\ln-b.
\end{equation}

Finally, by Proposition~\ref{pr:th5-1} we have
\begin{gather}
\np-\nl\ge(1+b)M_1,
\label{eq8-lem9-3}
\\
\pv-\pn\ge(1+b)M_2,
\label{eq9-lem9-3}
\\
\vp-\vl\ge(1+b)M_3,
\label{eq10-lem9-3}
\\
\lv-\ln\ge(1+b)M_4.
\label{eq11-lem9-3}
\end{gather}
Indeed, if $b=0$, these inequalities immediately follow from the proposition.  
If $b = 1$,  $\Gamma$ has large $\mathbf{e}_2$-step $n/2\ge2$ and as $\mathbf 
e_1 \notin \Gamma$ due to the restriction $1\le a\le n-1$, we see that 
$\Gamma$ has large $\mathbf{e}_1$-step greater then or equal to 2.

We estimate $2N$ by means of \eqref{eq2-lem9-3}--\eqref{eq3-lem9-3}, and 
\eqref{eq7-lem9-3}:
\begin{multline*}
2N=\sum_{k=1}^42N_k+\sum_{k=1}^42M_k
\\
\le
n + 2 -b-\Ceil{\frac{\p-n}{2}}
+ \p
-\Ceil{\frac{\v-n}{2}}+s_2+\pn+2M_2
\\
-(\np-\nl)
-(\vp-\vl)-(\lv-\ln)+2M_1+2M_3+2M_4
.
\end{multline*}
Dropping the ceilings, using~\eqref{eq4-lem9-2} and 
\eqref{eq8-lem9-3}--\eqref{eq11-lem9-3} and subsequently estimating $M_k \le 
1$, we obtain
\begin{equation}
\label{eq12-lem9-3}
2N \le 3n + \frac92 - 4b
+ \left( -\frac{\v}{2} + s_2 + \pn + 2M_2 \right)
.
\end{equation}

Let us estimate the term in parentheses on the right-hand side.  From~\eqref{eq4-lem9-3} and~\eqref{eq9-lem9-3} we get
\begin{equation*}
\v\ge\pv+\frac{s_2^2+s_2}{2},
\quad
\pn\le\pv-(1+b)M_2\le\pv-M_2
,
\end{equation*}
whence
\begin{equation}
\label{eq12.1-lem9-3}
-\frac{\v}{2} + s_2 + \pn + 2M_2
\le
\frac{\pv}{2} - \frac{s_2^2 - 3s_2}{4} + M_2
.
\end{equation}

It follows from assertion~(iv) of Lemma~\ref{lem9-2} that the vertex 
$(\pv,\p)$ of~$\m$ lies in the half-plane $x_1 \ge n$, so using assertion~(i) 
and~\eqref{eq4-lem9-2}, we get $\pv \le \p - 1 \le 2n - 2$.  Moreover, $M_2 
\le 1$ and $s_2^2 - 3s_2 \ge -2$, since $s_2$ is an integer, so 
from~\eqref{eq12.1-lem9-3} we obtain
\begin{equation*}
-\frac{\v}{2} + s_2 + \pn + 2M_2
\le
n + \frac 12
.
\end{equation*}
Combining this with~\eqref{eq12-lem9-3}, we get
\begin{equation*}
2N \le 4n + 5 - 4b.
\end{equation*}
Dividing both sides by 2 and taking floor, we obtain~\eqref{eq1-lem9-3}.
\end{proof}

\begin{lemma}
\label{lem9-4}
Suppose that $\m$ is a type~IV$_n$ $N$-gon, the segment $[(0,-n),(n,0)]$ 
splits $\m$, the vertices of $\m$ belong to a $(1, m)$-lattice $\Gamma$, where 
$m$ divides $n$, and $\Gamma$ has small $\mathbf{e}_1$-step and large 
$\mathbf{e}_2$-step greater than or equal to $2$.  Then
$$N\le 2n-2.$$
\end{lemma}
\begin{proof}
Applying Proposition~\ref{pr:slp} to the basis $(-\mathbf{e}_1,\mathbf{e}_2)$ 
and the slope $Q_1$ and to the basis $(-\mathbf{e}_1,-\mathbf{e}_2)$ and the 
slope $Q_2$, we obtain
\begin{gather}
2N_1\le \p-\np,
\label{eq:t4-lem2-1}
\\
2N_2\le \p-\vp.
\label{eq:t4-lem2-2}
\end{gather}

Applying Theorem~\ref{th3-8} to the frame $((n,n),\mathbf{e}_1,-\mathbf{e}_2)$ 
and the slope $Q_3$ and to the frame $(\mathbf{0},\mathbf{e}_1,-\mathbf{e}_2)$
and the slope $Q_4$, we obtain
\begin{gather}
2N_3\le \vl-\lv-1,
\label{eq:t4-lem2-3}
\\
2N_4\le \nl+\ln-1.
\label{eq:t4-lem2-4}
\end{gather}

By Proposition~\ref{pr:th5-1},
\begin{gather}
\np-\nl\ge2M_1,
\label{eq:t4-lem2-5}
\\
\pv-\pn\ge2M_2,
\label{eq:t4-lem2-6}
\\
\vp-\vl\ge2M_3,
\label{eq:t4-lem2-7}
\\
\lv-\ln\ge2M_4.
\label{eq:t4-lem2-8}
\end{gather}

Observe that $\nz$ is a sublattice of $\Gamma$.  Indeed, $\nz$ is a sublattice 
of $\z \times m\z$, since $m$ divides $n$, and the unimodular transformation 
mapping $\z \times m\z$ onto $\Gamma$ maps $\nz$ onto itself.  Thus, the point 
$(2n,0)$ belongs to~$\Gamma$.  So does the vertex $(\p,\pn)$.  Therefore, the 
small $\mathbf{e}_1$-step of $\Gamma$ divides the difference $2n-\p$, which is 
positive by~\eqref{eq4-lem9-2}.  Consequently, we obtain
\begin{equation}
\label{eq:t4-lem2-9}
\p\le 2n-2
.
\end{equation}

Now we use the above inequalities to estimate $N$.  
Summing~\eqref{eq:t4-lem2-1}--\eqref{eq:t4-lem2-4} and subsequently 
using~\eqref{eq:t4-lem2-5}--\eqref{eq:t4-lem2-8} and \eqref{eq:t4-lem2-9}, we 
obtain
\begin{multline*}
2N=\sum_{k=1}^42N_k+\sum_{k=1}^42M_k
\le
2\p
+(2M_1-(\np-\nl))
\\
+(2M_3-(\vp-\vl))
+(2M_4-(\lv-\ln))+2(M_2-1)
\\
\le 2\p\le 2(2n-2)
,
\end{multline*}
and the lemma follows.
\end{proof}

\begin{lemma}
\label{th3-9}
Let $n\ge 3$ be an integer and $Q$ be a slope with $N$ edges with respect to 
the basis $(-\mathbf{e}_1,\mathbf{e}_2)$.  Suppose that the vertices of $Q$ 
belong to a lattice~$\Gamma$ spanned by $\mathbf{e}_1+a\mathbf{e}_2$ and 
$n\mathbf{e}_2$, where $a$ is an integer.  Suppose that the frame 
$((n,0);-\mathbf{e}_1,\mathbf{e}_2)$ splits $Q$, the endpoints $\mathbf v = 
(v_1, v_2)$ and $\mathbf w = (w_1, w_2)$ of $Q$ satisfy
\begin{equation}
\label{eqA-th3-9}
v_1 < n, \ v_2 < 0,
\ n < w_1 \le 2n - 1, \ w_2 > 0
,
\end{equation}
the intersection of $Q$ with the half-plane $x_1 \ge n$ lies in the slab
\begin{equation}
\label{eqD-th3-9}
-n< x_2 - x_1 < 0,
\end{equation}
and $Q$ has a vertex $\mathbf u = (u_1, u_2)$ satisfying
\begin{equation}
\label{eqB-th3-9}
u_2 - u_1 \le - n - 1
.
\end{equation}
Then
\begin{equation}
\label{eqC-th3-9}
2N\le 2n-1- v_1
.
\end{equation}
\end{lemma}
\begin{proof}
Let $\mathbf v_0 = \mathbf v$, $\mathbf v_1$, \dots, $\mathbf v_N = \mathbf w$ 
be consecutive vertices of $Q$, $\eps_i = [\mathbf v_{i-1}, \mathbf v_i]$ and 
$\mathbf a_i = \mathbf v_i - \mathbf v_{i-1}$ be the edges and their 
associated vectors, where $i = 1, \dots, N$, and let $\mathbf v_i = (v_{i1}, 
v_{i2})$, $\mathbf a_i = (a_{i1}, a_{i2})$.  Of course, $v_{ij}$ and $a_{ij}$ 
are integers.  It follows from the definition of a slope that
\begin{gather}
a_{i1}\ge 1\qquad (i=1,\ldots,N);
\label{eq1-th3-9}
\\
a_{i2}\ge 1\qquad (i=1,\ldots,N);
\label{eq2-th3-9}
\\
\frac{a_{12}}{a_{11}}<\frac{a_{22}}{a_{21}}<\dots<\frac{a_{N2}}{a_{N1}};
\label{eq3-th3-9}
\\
v_{01}<v_{11}<\dots<v_{N1}.
\label{eq4-th3-9}
\end{gather}

Set
$$
E^{(j)}=\{\eps_i \colon a_{i1}=j\}, \ N^{(j)}=|E^{(j)}|\qquad(j=1,2,\dots).
$$
Of course, $E^{(j)}\neq\varnothing$ and $N^{(j)}\ne 0$ for finitely many~$j$.  
As the vectors $\mathbf a_i$ are distinct, we have
\begin{gather}
N=\sum_{j=1}^{\infty}N^{(j)}
,
\label{eq5-th3-9}
\\
w_1 - v_1 = \sum_{i = 1}^N a_{i1} = \sum_{j=1}^{\infty}jN^{(j)}
\label{eq6-th3-9}
.
\end{gather}

We claim that
\begin{equation}
\label{eq7-th3-9}
N^{(1)}\le1.
\end{equation}
It follows from~\eqref{eqA-th3-9} that $v_{N1}\ge n+1$ and taking into 
account~\eqref{eq1-th3-9}, we also obtain $v_{N-1,1}=v_{N1}-a_{N1}\le
v_N-1$.  Therefore, there exists a point $\mathbf y = (y_1, y_2) \in \eps_N$ 
such that $y_1=v_{N1}-1\ge n$.  Then
$$
\frac{a_{N2}}{a_{N1}}
=\frac{v_{N2}-y_2}{v_{N1}-y_1}=v_{N2}-y_2
.
$$
As the points $\mathbf v_N$ and $\mathbf y$ lie in the half-plane $x_1 \ge n$, 
they satisfy
$$-n< y_2 - y_1<0,\qquad
-n< v_{N2} - v_{N1} < 0,$$
whence $v_{N2} \le v_{N1} - 1$ and
$$\frac{a_{N2}}{a_{N1}}<(v_{N1} - 1 )-(y_1 
-n)
=n+(v_{N1} - y_1)-1=n.$$
Thus, according to~\eqref{eq3-th3-9},
\begin{equation}
\label{eq8-th3-9}
a_{i2}=\frac{a_{i2}}{a_{i1}}<n\qquad (\eps_i\in E^{(1)})
.
\end{equation}
As $\mathbf a_i \in \Gamma$, it is easily seen that possible values for  
$\mathbf{a}_i$ corresponding to $\eps_i\in E^{(1)}$ belong to the set 
$\{\mathbf{e}_1+(a+pn)\mathbf{e}_2\colon p\in\z\}$.  Only one vector of this 
set satisfies both~\eqref{eq2-th3-9} and~\eqref{eq8-th3-9}.  As the vectors 
$\mathbf a_i$ are distinct, we conclude that $E^{(1)}$ contains at most one 
edge, and~\eqref{eq1-th3-9} follows.

Having established all these auxiliary facts, we start 
proving~\eqref{eqC-th3-9}.  Assuming the converse, we have
$$2n-v_1-2N\le0,$$
or, equivalently,
$$(2n-1-w_1)+(1-N^{(1)})+\sum_{j=3}^{\infty}(j-2)N^{(j)}\le0,$$
where we have used~\eqref{eq5-th3-9} and~\eqref{eq6-th3-9}.  In view 
of~\eqref{eqA-th3-9} and \eqref{eq8-th3-9}, the three summands on the 
left-hand side are nonnegative.  Consequently, we obtain
\begin{gather}
w_1 =2n-1,
\label{eq9-th3-9}
\\
N^{(1)}=1,
\label{eq10-th3-9}
\\
N^{(j)}=0,\qquad(j=3,4,\ldots)
\label{eq11-th3-9}
\end{gather}
We must have $E^{(2)} \ne \varnothing$, for otherwise
$N^{(2)}=0$ and \eqref{eq10-th3-9}, \eqref{eq11-th3-9}, and \eqref{eq6-th3-9} 
would give $ w_1 - v_1 = 1 $, which together with \eqref{eq9-th3-9} implies
$$v_1 = w_1-1=2n-2>n,$$
in contradiction to~\eqref{eqA-th3-9}.  Set
$$i'=\max\{i\colon \eps_i\in E^{(2)}\}.$$
Let us show that if
\begin{equation}
\label{eq11,5-th3-9}
v_{i'-1,1}\ge n,
\end{equation}
then
\begin{equation}\label{eq12-th3-9}
a_{i2}\le n \qquad(\eps_i\in
E^{(2)}).
\end{equation}
Indeed, if~\eqref{eq11,5-th3-9} holds, according to~\eqref{eq4-th3-9} we have
$n\le v_{i'-1,1}<v_{i'1}$, so the vertices $\mathbf v_{i'-1}$ and $\mathbf 
v_{i'}$ belong to the slab~\eqref{eqD-th3-9}, whence
\begin{gather*}
-n+1\le
v_{i'-1,2}
-
v_{i'-1,1}
\le -1,
\\
-n+1\le
v_{i'2}
-
v_{i'1}
\le -1,
\end{gather*}
Then
\begin{gather*}
\frac{a_{i'2}}{a_{i'1}}=\frac{v_{i'2}-v_{i'-1,2}}{2}
\le\frac{(v_{i'1}-1)-(v_{i'-1,1}-n+1)}{2}
=\frac{a_{i'1}-2+n}{2}=\frac n2.
\end{gather*}
This and~\eqref{eq3-th3-9} imply \eqref{eq12-th3-9}.

Assume that $n\ge 4$.  Let us show that in this case
\begin{equation}
\label{eq13-th3-9}
N^{(2)}=1.
\end{equation}
To this end let us estimate $v_{i'-1,1}$.  If $\eps_N\in E^{(2)}$, we have $i' 
= N$ and by virtue of \eqref{eq9-th3-9} we get
$$v_{i'-1,1}=v_{N1}-a_{N1}=2n-3\ge n.$$
Otherwise, $\eps_N\in E^{(1)}$, then $\eps_{N-1}\in E^{(2)}$, since 
$|E^{(1)}|=1$.  Moreover, $i'=N-1$ and
\begin{gather*}
v_{N-1,1}=v_{N1}-a_{N1}=2n-2,
\\
v_{i'-1,1}=v_{N-1,1}-a_{N-1,1}=2n-4\ge n
.
\end{gather*}
Thus, in any case we have~\eqref{eq11,5-th3-9}, so~\eqref{eq12-th3-9} holds.

As the vectors $\mathbf a_i$ associated with edges from $E^{(2)}$ belong to 
$\Gamma$, it is easily seen that they have the form $\mathbf a_i = (2, 
2a+pn)$, where $p$ is an integer.  Clearly, only one vector of this form 
satisfies~$1 \le a_{i2} \le n$ and as the vectors $\mathbf a_i$ are distinct, 
we see that $E^{(2)}$ contains at most one edge.  Thus,~\eqref{eq13-th3-9} is 
proved.

From \eqref{eq6-th3-9}, \eqref{eq10-th3-9}, \eqref{eq11-th3-9}, and 
\eqref{eq13-th3-9} it follows that
$$w_1 - v_1=3,$$
and according to~\eqref{eq9-th3-9},
$$v_1 = w_1-3=2n-4\ge n,$$
which contradicts \eqref{eqA-th3-9}.  The contradiction proves the lemma for 
$n\ge4$.

Now assume that $n=3$.  Let us prove the following assertions:
\begin{enumerate}[(a)]
\em
\item
The numbers $a_{i2}/a_{i1}$ $(i=1,\ldots,N)$ are positive integers or 
half-integers not exceeding $2$.
\item
The numbers $a_{i2}/a_{i1}$ $(i=1,\ldots,N-1)$ are positive integers or 
half-integers not exceeding $3/2$.
\end{enumerate}

Let us show (a).  Fix $i\in \{1,\ldots,N\}$.  If $\eps_i\in E^{(1)}$, then
$a_{i2}/a_{i1} = a_{i2}$ is a positive integer (according 
to~\eqref{eq1-th3-9}) not exceeding $2$ according to~\eqref{eq8-th3-9}.  
Assume that $\eps_i\in E^{(2)}$, then $a_{i2}/a_{i1} = a_{i2}/2$ is positive 
and either an integer or a half-integer.  Let us show that it cannot be 
greater then $2$.  If $\eps_N\in E^{(1)}$, by the above we have
$a_{N2}/a_{N1}\le 2$ and the required estimate follows from~\eqref{eq3-th3-9}.  
On the other hand, if $\eps_N\in E^{(2)}$, then $i'=N$ and by virtue 
of~\eqref{eq9-th3-9} we have
$$v_{i'-1,1}=v_{N1}-2=3,$$
i.~e.~\eqref{eq11,5-th3-9} holds, and the required estimate follows 
from~\eqref{eq12-th3-9}.  Assertion~(a) is proved.

Assertion (b) is a corollary of (a), given that by virtue 
of~\eqref{eq3-th3-9}, the ratio $a_{i2}/a_{i1}$ cannot attain its maximum at 
$i < N$.

By hypothesis, the slope~$Q$ has a vertex $(u_1, u_2)$ satisfying $u_2 < u_1 - 
n = u_1 - 3$.  Set
$$i_0=\max\{i\colon v_{i2} < v_{i1} - 3 \}.$$
Observe that \emph{a priori} $i_0 < N$.  Then we have $v_{i_0 + 1, 2} \ge 
v_{i_0 +1, 1} - 3$ and therefore,
\begin{multline*}
\frac{a_{i_0 + 1, 2}}{a_{i_0 + 1, 1}}
=\frac{v_{i_0+1, 2} - v_{i_0, 2}}{a_{i_0+1,1}}
>\frac{(v_{i_0+1, 1} - 3) - (v_{i_0, 1} - 3)}{a_{i_0+1,1}}
\\
=\frac{v_{i_0+1, 1} - v_{i_0, 1}}{a_{i_0+1,1}}
= 1
.
\end{multline*}
We must have $v_{i_0} \le 2$, because otherwise we would have $v_{i_0} \ge 3 = 
n$, and by hypothesis, $v_{i_02} < v_{i_01} - 3$, contrary to the definition 
of $i_0$.  Thus,
\begin{equation*}
v_{i_0+1, 1} = v_{i_01} + a_{i_0+1,1} \le 4
.
\end{equation*}
This and~\eqref{eq9-th3-9} imply $i_0+1<N$.  Then assertion~(b) and the 
inequality $a_{i_0+1,2}/a_{i_0+1,1} > 1$ proved above yield 
$a_{i_0+1,2}/a_{i_0+1,1} = 3/2$, which is only possible if
$\mathbf{a}_{i_0+1}=(2,3)$.

Let $\eps_{i_1}\in E^{(1)}$, then by assertion~(a), we have either 
$\mathbf{a}_{i1}=(1,1)$ or $\mathbf{a}_{i1}=(1,2)$.  By 
Proposition~\ref{pr:th1-1}, in both cases the vectors $\mathbf{a}_{i_1}$ and 
$\mathbf{a}_{i_0+1}$ form a basis of $\zz$, which is impossible, since they 
belong to its proper sublattice~$\Gamma$.  The contradiction proves the lemma 
in the case $n=3$.
\end{proof}

\begin{lemma}
\label{lem9-5}
Suppose that $\m$ is a type~IV$_n$ $N$-gon, the segment $[(0,-n),(n,0)]$ 
splits $\m$, and the vertices of $\m$ belong to a lattice $\Gamma \subset \zz$
having the basis $(\mathbf{e}_1+a\mathbf{e}_2$, $n\mathbf{e}_2)$, где
$1\le a\le n-1$.  Then
\begin{equation}
\label{eq1-lem9-5}
N\le 2n-2.
\end{equation}
\end{lemma}
\begin{proof}
Lemma~\ref{lem9-2} ensures that we can apply Lemma~\ref{th3-9} to the 
slope~$Q_1$ and obtain
\begin{equation}
\label{eq3-lem9-5}
2N_1\le2n-1-\np.
\end{equation}

Set $\mathbf{f}_1=-\mathbf{e}_1$, $\mathbf{f}_2=-\mathbf{e}_2$, so that $Q_2$ 
is a slope with respect to $(\mathbf f_1, \mathbf f_2)$.  By hypothesis, the 
vectors $\mathbf{b}_1=-\mathbf{f}_1-a\mathbf{f}_2$,
$\mathbf{b}_2=-n\mathbf{f}_2$ form a basis of~$\Gamma$.  By 
Proposition~\ref{pr:th1-1}, the vectors 
$\mathbf{f}_1-(n-a)\mathbf{f}_2=-\mathbf{b}_1+\mathbf{b}_2$ and 
$n\mathbf{f}_2=-\mathbf{b}_2$ form a basis of $\Gamma$ as well, and as $1\le 
n-a\le n-1$, we apply Proposition~\ref{pr:slp} and conclude that there exists 
an integer~$s_2$ such that
\begin{gather}
2N_2\le \p-\vp+s_2,
\label{eq4-lem9-5}
\\
\v-\pv\ge\frac{2(n-a)+(s_2-1)n}{2}s_2.
\label{eq4.1-lem9-5}
\end{gather}
As $n-a\ge1$, \eqref{eq4.1-lem9-5} implies
\begin{equation}
\label{eq6-lem9-5}
\v-\pv\ge\frac{2+(s_2-1)n}{2}s_2.
\end{equation}

Applying Theorem~\ref{th3-8} to the slope~$Q_3$ and the 
frame~$((n,n),\mathbf{e}_1,-\mathbf{e}_2)$ and to the slope $Q_4$ and the 
frame $(\mathbf{0};\mathbf{e}_2,\mathbf{e}_1)$, we obtain
\begin{gather}
2N_3\le\vl-\lv-1,
\label{eq7-lem9-5}
\\
2N_4\le\nl+\ln-1.
\label{eq8-lem9-5}
\end{gather}

By assertion~(vi) of Lemma~\ref{lem9-2}, the points $(\p,\pn)$, $(\p,\pv)$, 
and $(\vp,\v)$ lie in the half-plane $x_1\ge n$, so by assertion~(v) of the 
same lemma we have
\begin{gather}
-n+1\le\pn - \p \le-1,
\label{eq9-lem9-5}
\\
-n+1\le\pv - \p\le-1,
\label{eq10-lem9-5}
\\
-n+1\le\v - \vp \le-1.
\label{eq11-lem9-5}
\end{gather}
From~\eqref{eq4-lem9-2} we also have
\begin{equation}\label{eq11,5-lem9-5}
\p\le2n-1.
\end{equation}

It is clear that $\mathbf{e}_1 \notin \Gamma$, so the large 
$\mathbf{e}_1$-step of $\Gamma$ cannot be less then 2.  It is easily seen that 
the large $\mathbf{e}_2$-step of~$\Gamma$ equals~$n$.  By 
Proposition~\ref{pr:th5-1}, we get
\begin{gather}
\np-\nl\ge2M_1,
\label{eq12-lem9-5}
\\
\pv-\pn\ge nM_1,
\label{eq13-lem9-5}
\\
\vp-\vl\ge2M_3,
\label{eq14-lem9-5}
\\
\lv-\ln\ge nM_4.
\label{eq15-lem9-5}
\end{gather}

Now we deduce a few implications of the inequalities.

Inequalities~\eqref{eq9-lem9-5} and~\eqref{eq10-lem9-5} yield
$$\pv-\pn\le n-2.$$
This and~\eqref{eq13-lem9-5} give
$$nM_2\le n-2,$$
which can only hold if
\begin{equation}\label{eq16-lem9-5}
M_2=0.
\end{equation}

Let us estimate the difference $\v-\pv$ from above using~\eqref{eq10-lem9-5}, 
\eqref{eq11-lem9-5}, and the evident inequality $\p\ge\vp$.  We have:
$$\v-\pv\le(\vp-1)-(\p-n+1)=n-2-(\p-\vp)\le n-2.$$
Comparing this with~\eqref{eq6-lem9-5}, we obtain
$$\frac{2+n(s_2-1)}{2}s_2\le n-2,$$
which can only hold if
$$s_2\le1.$$
This and~\eqref{eq4-lem9-5} give
\begin{equation}
\label{eq17-lem9-5}
2N_2\le\p-\vp+1.
\end{equation}

Now we estimate $N$ by means of \eqref{eq3-lem9-5}, \eqref{eq7-lem9-5}, 
\eqref{eq8-lem9-5}, \eqref{eq11,5-lem9-5}, \eqref{eq12-lem9-5}, and 
\eqref{eq14-lem9-5}--\eqref{eq17-lem9-5}.  We have:
\begin{multline*}
2N=\sum_{k=1}^42N_k+\sum^4_{k=1}2M_k
\le
2n-2+\p
+(2M_1-(\np-\nl))
\\
+(2M_3-(\vp-\vl))+(nM_4-(\lv-\ln))-(n-2)M_4
\\
\le 4n-3.
\end{multline*}
Dividing by 2 and taking floor, we obtain~\eqref{eq1-lem9-5}.
\end{proof}

Now we are in position to prove Theorem~\ref{subtheorem-c} for type IV$_n$ 
polygons split by the segment $[(0,-n),(n,0)]$.

\begin{lemma}
\label{lem9-6}
Theorem~\ref{subtheorem-c} holds for type IV$_n$ polygons split by the segment 
$[(0,-n),(n,0)]$.
\end{lemma}
\begin{proof}
Let $\m$ be a $N$-gon satisfying the hypothesis of the lemma.

Lemma~\ref{lem9-3} grants the estimate
$$N\le2n+2
.
$$

Assume that $n$ is even and that the vertices of $\m$ belong to a 
$(1,n/2)$-lattice $\Gamma$.  Let us show that
\begin{equation}\label{eq5-lem9-6}
N\le2n.
\end{equation}

Let $s_1$ be the small $\mathbf{e}_1$-step of~$\Gamma$ and $S_2$ be its large 
$\mathbf{e}_2$-step.  By Proposition~\ref{pr:th1-11}, we have
\begin{equation}
\label{eq6-lem9-6}
s_1S_2=\frac{n}{2}.
\end{equation}

First, assume that $s_1=1$.  Then by Proposition~\ref{pr:th1-12}, $\Gamma$ 
admits a basis of the form $(\mathbf{e}_1+a\mathbf{e}_2,(n/2)\mathbf{e}_2)$, 
where $0\le a\le n/2-1$.  If $a\ge1$, the estimate \eqref{eq5-lem9-6} follows 
from Lemma~\ref{lem9-3}.  Assume that $a=0$.  It is easily seen that in this 
case all the points of~$\Gamma$ lie on the lines $x_2=(n/2)r$, where $r\in\z$.  
In particular, inequalities~\eqref{eq5-lem9-2} and \eqref{eq7-lem9-2} become
\begin{gather*}
\v=\frac{3n}{2},
\quad
\lv=\frac{n}{2}.
\end{gather*}
Consider the vertices $\mathbf{w_1}=(\vp,3n/2)$ and $\mathbf{w_2}=(\l,n/2)$ 
of~$\m$.  By~\eqref{eq4-lem9-2} and~\eqref{eq6-lem9-2}, their first components 
satisfy
$$n<\vp<2n,\quad \l<0.$$
Taking into account that $\m$ has such vertices as well as a common point with 
the segment $[(n,0),(n,n)]$, it is not hard to check that $\m$ and the lines
$$x_2=\frac12 x_1+\frac n2,\ x_1=n$$
satisfy the hypothesis of Lemma~\ref{Helly}.  Consequently, $\m$ contains the 
point $(n,n)\in\nz$, which is impossible.  Thus, we cannot have $a=0$ and 
estimate \eqref{eq5-lem9-3} is proved for the case $s_1=1$.

Now assume that $s_1\ge2$.  If additionally $S_2\ge2$, Lemma~\ref{lem9-4} 
provides an even stronger estimate than~\eqref{eq5-lem9-6}.  Assume that 
$S_2=1$, then \eqref{eq6-lem9-6} gives $s_1=n/2$.  Consequently, all the 
points of $\Gamma$ belonging to the slab $-n+1\le x_1\le 2n-1$, which contains 
$\m$, lie on the five lines
\begin{equation}
\label{eq7-lem9-6}
x_1=\frac{nr}{2} \qquad (r=0,\pm1,\pm2).
\end{equation}
Thus, $\m$ has no more than 10 vertices, and \eqref{eq5-lem9-6} is true for 
$n\ge 6$ (remember that we are considering even $n$ at the moment).  Assume 
that $n=4$.  The vertex $(\vp,\v)$ lies on one of the 
lines~\eqref{eq7-lem9-6}, so taking into account~\eqref{eq4-lem9-2}, we see 
that necessarily $\vp=6$.  Inequalities~\eqref{eq4-lem9-2} 
and~\eqref{eq5-lem9-2} imply that
$$4 < \v \le \vp - 1 =5,$$
so necessarily $\v=5$ and $\mathbf{u}_1=(6,5)$ is a vertex of $\m$.  Likewise, 
\eqref{eq6-lem9-2} can only hold if $\l=-2$.  Consequently, 
$\mathbf{u}_2=(-2,\lv)$ is a vertex of $\m$, and according 
to~\eqref{eq7-lem9-2} we have
$$1\le \lv\le3.$$
Given that $\m$ has the vertices $\mathbf{u}_1$ and $\mathbf{u}_2$ with said 
properties as well as a common point with the segment
$$[(n,0),(n,n)]=[(4,0),(4,4)],$$
it is not hard to check that $\m$ and the lines
$$x_2=\frac12 x_1 +2,\ x_1=4$$
satisfy the hypotheses of Lemma~\ref{Helly}.  Consequently, $\m$ contains the 
point $(4,4)\in4\zz=\nz$, which is impossible.  The contradiction means that 
the vertices of a type IV$_4$ polygon cannot belong to a lattice $\Gamma$ 
having said properties.

Thus, \eqref{eq5-lem9-6} holds for any $(1, n/2)$-lattice $\Gamma$.

Finally, suppose that the vertices of $\m$ belong to a $(1,n)$-lattice 
$\Gamma$ (now there is no need to assume that $n$ is even).  Let us prove that
\begin{equation}
\label{eq8-lem9-6}
N\le2n-2.
\end{equation}
Note that as $\m$ has the vertex $(\vp,\v)$ satisfying~\eqref{eq4-lem9-2} 
and~\eqref{eq5-lem9-2}, it is clear that the both the small 
$\mathbf{e}_1$-step and the small $\mathbf{e}_2$-step of~$\Gamma$ are not 
equal to~$n$.  By Proposition~\ref{pr:th1-11}, the product of the small 
$\mathbf{e}_1$-step and the large $\mathbf{e}_2$-step of $\Gamma$ equals $n$, 
so $\Gamma$ has large $\mathbf{e}_2$-step different from 1.  Likewise, 
$\Gamma$ has large $\mathbf{e}_2$-step different from 1 as well.

Suppose that the small $\mathbf{e}_1$-step of~$\Gamma$ equals~1.  Then by 
Proposition~\ref{pr:th1-12} the lattice admits a basis of the form 
$(\mathbf{e}_1+a\mathbf{e}_2, n\mathbf{e}_2)$, where $0\le
a\le n-1$.  The equality $a=0$ is impossible, as the large $\mathbf{e}_1$-step 
of the lattece is not~1.  Consequently, we can apply Lemma~\ref{lem9-5}, which 
gives~\eqref{eq8-lem9-6}.

Otherwise, the small $\mathbf{e}_1$-step of~$\Gamma$ is greater than~1, and as 
its large $\mathbf{e}_2$-step is greater than~1 as well, we can apply 
Lemma~\ref{lem9-4} and obtain~\eqref{eq8-lem9-6}.
\end{proof}

\begin{proof}[Proof of Theorem~\ref{subtheorem-c} for type~IV$_n$ polygons.]
Let $\m$ be an arbitrary type~IV$_n$ polygon.  If the segment
$[(0,-n),(n,0)]$ splits it, we complete the proof by evoking 
Lemma~\ref{lem9-6}.  Otherwise it suffices to show that there is an affine 
automorphism of $\zz$ mapping $\m$ on a type~II$_n$ or a type~III$_n$ polygon, 
as the required estimates have already been proved for those kinds of 
polygons.

Define the automorphism $\varphi$ by
$$\varphi(x_1,x_2)=(-x_1+x_2+n,x_2).$$

By definition, the segments $[\mathbf{0},(n,0)]$ and $[(n,n),(2n,n)]$ split  
the polygon~$\m$, and so does $[\mathbf{0},(n,n)]$ by virtue of 
Lemma~\ref{lem9-1}.  Consequently, the images of those segments under 
$\varphi$---i.~e., the segments $[(n,0),\mathbf{0}]$, $[(n,n),(0,n)]$, 
$[(n,n),(n,2n)]$---split $\varphi(\m)$.  If $\m$ is also split by 
$[(n,0),(2n,n)]$, then $\varphi(\m)$ is split by $[\mathbf{0},(0,n)]$, and 
consequently, $\varphi(\m)$ is a type II$_n$ polygon.  Otherwise, the line 
$x_1-x_2=n$ does not split $\m$, so the line $x_1 = 0$ does not split 
$\varphi(\m)$ either, and the latter is a type III$_n$ polygon.
\end{proof}

\section{Type V and Va polygons}
\label{sec:v}

\subsection{Main results}

In this section we prove that for any type~V$_n$ polygon there exists an 
affine automorphism of~$\nz$ mapping it on a type~III$_n$ or a type~Va$_n$ 
polygon, the latter to be defined presently.  We find certain bounds for the 
number of vertices of type~Va polygons, which are not sufficient, however, to 
prove Theorem~\ref{subtheorem-c} for this class of polygons.  We revisit 
type~Va polygons in Section~\ref{sec:last}, establishing the missing estimate 
and completing the proof of Theorem~\ref{subtheorem-c}.

Fix an integer $n \ge 3$.

We will denote by $\Delta_n$ the triangle with the vertices $\mathbf 0$, $(2n, 
0)$, and $(0, 2n)$.  It is easy to check that $\Delta_n$ is given by the 
following system of linear inequalities:
\begin{equation}
\label{eq*-11}
\left\{
\begin{array}{c}
x_1\ge0,
\\
x_2\ge0,
\\
x_1+x_2\le 2n.
\end{array}
\right.
\end{equation}

\begin{definition}
We say that $\m$ is a \emph{type V$a_n$ polygon}, if it is free 
of~$\nz$-points and lies in~$\Delta_n$.
\end{definition}

Two following lemmas are the main results of the section.

\begin{lemma}
\label{lem:th11-1}
For any type~V$_n$ polygon there exists an affine automorphism of $\nz$ 
mapping it onto a type~III$_n$ or a type~Va$_n$ polygon.
\end{lemma}

The proof is given in Section~\ref{sec:lift}.

\begin{lemma}
\label{lem:th11-2}
Suppose that $\m$ is a type~Va$_n$ $N$-gon; then
$$N\le 2n+2,$$
and if the vertices of $\m$ belong to a $(1,n)$-lattice, then
$$N\le2n-2.$$
\end{lemma}

The proof is given in Section~\ref{sec:v-pf2}.

\subsection{The lift}
\label{sec:lift}

Let $\m$ be an integer polygon free of $\nz$-points, where $n \ge 3$ is an 
integer.  Assume that the segments
$[\mathbf{0},(-n,0)]$ and $[\mathbf{0},(0,n)]$ split $\m$.  In particular, 
$\m$ can be any type~V$_n$ polygon.  Given $a\in\z$, consider the unimodular 
transformation
$$A_a
=
\begin{pmatrix}
1 & 0 \\
-a & 1
\end{pmatrix}
$$
and the polygon $\m_a=A_a\m$.
\begin{lemma}
\label{lem11-01}
The set of such $a \in \z$ that~$\m_a$ is split by the segment 
$[\mathbf{0},(-n,0)]$ is nonempty and has a nonnegative maximal element.
\end{lemma}
\begin{proof}
Obviously, $\m_0=\m$, so the set in question contains 0 and its maximal 
element, if it exists, is nonnegative.  To prove the lemma, it remains to show 
that the set is bounded from above, i.~e.\ that the segment 
$[\mathbf{0},(-n,0)]$ does not split the polygon $\m_a$ for large $a$.

As $\m$ does not contain the point $\mathbf{0}\in\nz$, there exists a linear 
form $\ell(x_1, x_2) = \alpha_1 x_1 + \alpha_2 x_2$ such that
\begin{equation}
\label{eq1-lem11-1}
\ell(\mathbf{x}) > 0,\  \mathbf x\in\m.
\end{equation}
Choosing points $\check{\mathbf x} \in \m \cap [\mathbf 0, (0, n)]$ and $\hat 
{\mathbf x} \in \m \cap [\mathbf 0, (-n, 0)]$, so that  $\check{\mathbf 
x}=(0,\check{x}_2)$, $\check{x}_2>0$, and $\hat{\mathbf x}=(\hat{x}_1,0)$, 
$\hat{x}_1<0$ and computing $\ell(\check{\mathbf x})$ and $\ell(\hat{\mathbf 
x})$, we see that in view of \eqref{eq1-lem11-1},
\begin{equation}\label{eq2-lem11-1}
\alpha_1<0,\ \alpha_2 > 0.
\end{equation}

Fix an integer~$a$ such that
\begin{equation}\label{eq4-lem11-1}
a\ge-\frac{\alpha_1}{\alpha_2}.
\end{equation}
Consider the linear form
$$\tilde \ell(\tilde x_1, \tilde x_2)=(\alpha_1+a\alpha_2)\tilde x_1+\alpha_2 
\tilde x_2.$$
It is easy to check that $\ell(\mathbf x) = \tilde \ell(\tilde{\mathbf x})$ 
whenever $\mathbf x = A_a^{-1}\tilde{\mathbf x}$.  In particular, if 
$\tilde{\mathbf x} \in \m_a$, we have $\mathbf x = A_{a_0}^{-1} \tilde{\mathbf 
x} \in \m$, and according to \eqref{eq1-lem11-1}, we obtain
\begin{equation}
\label{eq5-lem11-1}
\tilde \ell(\tilde{\mathbf{x}}) > 0,\quad \tilde{\mathbf{x}}\in\m_a.
\end{equation}
On the other hand, if $\tilde{\mathbf{x}}=(\tilde x_1,0)\in[0,(-n,0)]$, then 
$\tilde x_1 \le 0$, and
\begin{equation}
\label{eq6-lem11-1}
\tilde \ell(\tilde{\mathbf{x}})=(\alpha_1+a\alpha_2)\tilde x_1<0,\quad
\tilde{\mathbf{x}}\in[\mathbf{0},(-n,0)]
\end{equation}
according to the choice of $a$.  Comparing~\eqref{eq5-lem11-1} 
and~\eqref{eq6-lem11-1}, we see that $\m_a$ has no common points with the 
segment $[\mathbf{0},(-n,0)]$.  This is true for any~$a$ satisfying 
\eqref{eq4-lem11-1}, so the set in question is bounded from above, as claimed. 
\end{proof}

Let $a_0 \ge 0$ be the greatest integer such that $\m_{a_0}$ is split by the 
segment $[\mathbf{0},(-n,0)]$.  We say that the polygon $\tm=\m_{a_0}$ is the 
\emph{lift} of $\m$ and that $A_{a_0}$ is the \emph{lift transformation} of 
$\m$.

\begin{lemma}
\label{lem11-02}
Let that $\tm$ be the lift of $\m$; then the segment $[\mathbf{0},(0,n)]$ 
splits $\tm$ and the segment $[\mathbf{0},(-n,-n)]$ does not.  If the 
segment~$[(0, n), (n, 2n)]$ does not split~$\m$, it does not split~$\tm$ 
either.
\end{lemma}
\begin{proof}
The line $x_1 = 0$ is invariant under the lift transformation and since the 
segment $[\mathbf{0},(0,n)] = A_{a_0}[\mathbf{0},(0,n)]$ splits $\m$, it 
splits $A_{a_0}\m = \tm$ as well.

By the definition of $a_0$, the segment $[\mathbf 0, (-n,0)]$ does not split 
the polygon $A_{a_0+1}\m = A_1 \tm$.  Consequently, the segment $[\mathbf 0, 
(-n, -n)] = A_1^{-1}[\mathbf 0, (-n,0)]$ does not split $\tm$, as claimed.

Suppose that the segment $[(0, n), (n, 2n)]$ does not split~$\m$.  The the 
intersection $\m \cap \{0 \le x_1 \le n\}$ lies in the half-plane $x_2 \le x_1 
+ n$.  It suffices to check that the intersection $\tm \cap \{0 \le x_1 \le 
n\}$ lies in the same half-plane.  Indeed, let $(\hat x_1, \hat x_2) \in \tm$ 
and $0 \le \hat x_1 \le n$; then $\hat x_1 = x_1$ and $\hat x_2 = -a_0 x + 
x_2$ for some $(x_1, x_2) \in \m \cap \{0 \le x_1 \le n\}$, so $\hat x_2 \le 
x_2 \le x_1 + n = \hat x_1 + n$, as claimed.
\end{proof}

\begin{lemma}
\label{lem11-03}
Suppose that $\tm$ is the lift of $\m$; then $\n(\tm)\ge\n(\m)$ and 
$\n(\tm)=\n(\m)$ if and only if $\m = \tm$.
\end{lemma}
\begin{proof}
If $a_0 = 0$, we have $\m=\tm$, so $\n(\tm)=\n(\m)$.  It remains to show that 
\begin{equation}\label{eq6-lem11-1}
a_0\ge1
\end{equation}
implies
\begin{equation}\label{eq7-lem11-1}
\n(\tm)>\n(\m).
\end{equation}
As the segments $[\mathbf{0},(-n,0)]$ and $[\mathbf{0},(0,n)]$ split both~$\m$ 
and~$\tm$ , by Proposition~\ref{pr:th5-3} the frame 
$(\mathbf{0};-\mathbf{e}_1,\mathbf{e}_2)$ splits the slopes~$Q_1(\m)$ 
and~$Q_1(\tm)$, whence
$$\np(\m)\le-1,\quad\np(\tm)\le-1.$$
Thus,
$$\n(\m)=\min\{x_2\colon (x_1,x_2)\in\m,\;x_1\le
-1\},$$
$$\n(\tm)=\min\{x_2\colon (x_1,x_2)\in\tm,\;x_1\le
-1\}.$$
Using these representations and \eqref{eq6-lem11-1}, we get
\begin{align*}
\n(\tm)&=\min\{\hat x_2\colon (\hat x_1,\hat x_2)\in\tm,\,\hat x_1\le-1\}
\\
&=\min\{-a_0x_1+x_2\colon (x_1,x_2)\in\m,\,x_1\le-1\}
\\
&\ge\min\{x_2+1\colon (x_1,x_2)\in\m,\,x_1\le-1\}=\n(\m)+1,
\end{align*}
so \eqref{eq7-lem11-1} is proved.
%
%
\end{proof}
\begin{lemma}
\label{lem11-2}
Let~$\m$ be a type~V$_n$ polygon and $\tm$ be its lift.  Then either $\tm$ is 
a type~V$_n$ polygon, or the translation of $\tm$ by the vector $(n, 0)$ is a 
type~III$_n$ polygon.
\end{lemma}
\begin{proof}
Let $T$ be the translation by the vector $(n, 0)$.  Note that $\tm$ and~$T\tm$ 
are obtained by applying affine automorphisms os $\nz$ to~$\m$, so they are 
free of points of this lattice.

The polygon $\tm$ is split by the segments $[\mathbf{0},(0,n)]$ (by 
Lemma~\ref{lem11-02}) and $[\mathbf{0},(-n,0)]$ (by the definition of lift), 
but not by the line $x_1 = -n$ (because by the definition of a type~V$_n$ 
polygon this line does not split $\m$ and it is invariant under the lift 
transformation).  Assume for a moment that the segment $[(0,n),(-n,n)]$ splits 
$\tm$. Then the segments
\begin{gather*}
[\mathbf{0},(n,0)]=T[(-n,0),\mathbf{0}],
\\
[(n,0),(n,n)]=T[\mathbf{0},(0,n)],
\\
[(0,n),(n,n)]=T[(-n,n),(0,n)]. 
\end{gather*}
split $T\m$, while the line $x_1 = 0$, being the image of $x_1 = -n$ under 
$T$, does not.  Consequently, $T\m$ is a type~III$_n$ polygon.

It remains to show that if the segment $[(0,n),(-n,n)]$ does not split $\tm$, 
the latter is a type~V$_n$ polygon.  We know already that the segments 
$[\mathbf{0},(-n,0)]$ and $[\mathbf{0},(0,n)]$ split $\tm$, while the line 
$x_1 = -n$ does not, so we only need to show that the line $x_2 = n$ does not 
split $\tm$ either, or, equivalently, that $\v(\tm) \le n$.

As $\tm$ lies in the half-plane $x_1 \ge -n$ and the segment $[(-n,n),(0,n)]$ 
does not split $\tm$, it is clear that
\begin{equation}
\label{eq1-lem11-2}
\max\{x_2 \colon (x_1, x_2) \in \tm, x_1 \le 0\} \le n
.
\end{equation}
On the other hand,
\begin{multline}
\label{eq2-lem11-2}
\max\{x_2\colon (x_1,x_2)\in\tm, x_1 \ge 0\}
\\
=\max\{-a_0x_1'+x_2'\colon (x_1',x_2')\in\m,x_1\le 0\}
\\
\le \max\{x_2'\colon (x_1',x_2')\in\m,x_1\le 0\} \le \v(\m) \le n
.
\end{multline}
Estimates~\eqref{eq1-lem11-2} and~\eqref{eq2-lem11-2} imply that $\v(\tm) \le 
n$, as claimed.
\end{proof}
\begin{proof}[Proof of Lemma~\ref{lem:th11-1}]
Take a type~V$_n$ polygon~$\m_0$, and let $\tm_0$ be its lift.  If the 
translation by the vector $(n,0)$ maps $\tm_0$ onto a type~III$_n$ polygon, we 
are done.  Otherwise, by Lemma~\ref{lem11-2}, $\tm_0$ is a type~V$_n$ polygon.  
Let $\m_0'$ be the reflection of $\tm_0$ about the line $x_1 + x_2 = 0$.  It 
is easy to check that it is again a type~V$_n$ polygon.  Let $\tm_0'$ be its 
lift.  As before, either the translation of $\tm_0'$ by $(n,0)$ is a 
type~III$_n$ polygon and we are done, or $\tm_0'$ is a type~V$_n$ polygon, in 
which case we define the type~V$_n$ polygon $\m_1$ to be the reflection of 
$\tm_0'$ about the line $x_1 + x_2 = 0$.

Iterating this procedure, we either find an affine automorphism of $\nz$ 
mapping $\m_0$ onto a type~III$_n$ polygon, or construct the sequences of 
type~V$_n$ polygons $\{\m_k\}$, $\{\tm_k\}$, $\{\m_k'\}$, and
$\{\tm_k'\}$.  In the latter case consider the sequence of integers
$\{\n(\m_k)\}_{k=0}^{\infty}$.  As $\m_k$ are type~V$_n$ polygons, it is 
easily seen that the the members of this sequence are negative (this follows 
e.~g.\ from the fact that by Proposition~\ref{pr:th5-3} the frame 
$(\mathbf{0};-\mathbf{e}_1,-\mathbf{e}_2)$ splits any type~V$_n$ polygon).  
Observe that the sequence increases.  Indeed, it is easy to check that
\begin{equation*}
\n(\m_{k+1}) = - \p(\tm_k') = - \p(\m_k') = \n(\tm_k)
;
\end{equation*}
furthermore, by Lemma~\ref{lem11-03} we have
$$\n(\tm_k)\ge\n(\m_k).$$
Thus, we see that the sequence $\{\n(\m_k)\}$ increases; moreover, we have 
$\n(\m_{k+1}) = \n(\m_k)$ if and only if $\n(\tm_k) = \n(\m_k)$, which by 
Lemma~\ref{lem11-03} is equivalent to $\tm_k = \m_k$.

The sequence of integers $\{\n(\m_k)\}$ increases and is bounded from above, 
so it stabilises.  We show in the same way that the sequence $\{\n(\m_k')\}$ 
stabilises, too.   Consequently, there exists $k_0$ such that $\m_{k_0} = 
\tm_{k_0}$ and $\m'_{k_0} = \tm'_{k_0}$.  Then also $\m_{k_0} = \m_{k_0 + 1}$.  
Set $\tm = \m_{k_0}$.

We claim that $\tm$ lies in the triangle $\Delta$ having the vertices $(-n, 
-n)$, $(-n, n)$, and $(n, n)$, which is the solution set of the system
\begin{equation*}
\left\{
\begin{array}{l}
x_1 \ge -n,
\\
x_2 \le n,
\\
x_1 - x_2 \le 0
.
\end{array}
\right.
\end{equation*}
Since $\tm$ is a type~V$_n$ polygon, it lies in the angle
\begin{equation*}
\left\{
\begin{array}{l}
x_1 \ge -n,
\\
x_2 \le n
.
\end{array}
\right.
\end{equation*}
The intersection of the line $x_1 - x_2 = 0$ with this angle is the segment 
$[(-n, -n), (n, n)]$, so we only need to show that neither of the segments 
$I_1 = [(-n, -n), \mathbf 0]$ and $I_2 = [\mathbf 0, (n, n)]$ splits~$\tm$.  
It the case of the former this is true by Lemma~\ref{lem11-02}, as~$\tm$ is 
the lift of~$\m_{k_0}$.  Likewise, $I_1$ does not split $\tm_{k_0}'$, 
so~$I_2$, being the reflection of~$I_1$ about the line $x_1 + x_2 = 0$, does 
not split $\m_{k_0+1} = \tm$, as claimed.

By construction, $\tm = B\m$, where $B$ is a unimodular transformation.  The 
affine automorphism of $\nz$ defined by
$$\psi(x_1,x_2)=(x_1+n,-x_2+n)$$
maps $\Delta$ onto $\Delta_n$.  Consequently, the polygon $\psi(B\m)$ lies 
in~$\Delta_n$, i.~e.\ $\varphi=\psi B$ is the required automorphism.
\end{proof}

\subsection{Bounds on the number of vertices of type~Va$_n$ polygons}
\label{sec:v-pf2}

Here we establish a few estimates of the number of vertices of type~Va$_n$ 
polygons and eventually prove Lemma~\ref{lem:th11-2}.

\begin{lemma}\label{lem11-3}
Suppose that $\m$ is a type Va$_n$ polygon, the frame 
$((n,n);-\mathbf{e}_2,-\mathbf{e}_1)$ splits the slope~$Q_2$ and forms small 
angle with it and either
\begin{equation}
\label{eqA-lem11-3}
\np\le n
\end{equation}
or
\begin{equation}
\label{eqB-lem11-3}
\np\ge n+1,\qquad \lv\le n.
\end{equation}
Then
\begin{equation}
\label{eq1-lem11-3}
N\le 2n+2.
\end{equation}
\end{lemma}
\begin{proof}
As $Q_1$ is a slope with respect to the basis $(\mathbf{e}_2,-\mathbf{e}_1)$, 
by Proposition~\ref{pr:slp} there exists an integer $s_1$ such that
\begin{gather}
2N_1\le \pn-\n+s_1,
\label{eq2-lem11-3}
\\
\p-\np\ge\frac12 s_1(s_1+1),
\label{eq3-lem11-3}
\\
0\le s_1\le N_1.
\label{eq4-lem11-3}
\end{gather}
The same proposition applied to $Q_3$ and $(\mathbf e_1,-\mathbf e_2)$ ensures 
the existence of an integer~$s_3$ such that
\begin{gather}
2N_3\le\vl-\l+s_3,
\label{eq6-lem11-3}
\\
\v-\lv\ge\frac12 s_3(s_3+1),
\label{eq7-lem11-3}
\\
0\le s_3\le N_3.
\label{eq8-lem11-3}
\end{gather}

As $Q_4$ is a slope with respect to the bases $(\mathbf{e}_1,\mathbf{e}_2)$ 
and $(\mathbf{e}_2,\mathbf{e}_1)$, by the same proposition there exist 
integers $s$ and $s'$ such that
\begin{gather}
2N_4\le \nl-\l+s,
\label{eq9-lem11-3}
\\
\ln-\n\ge\frac12s(s+1),
\label{eq10-lem11-3}
\\
0\le s\le N_4,
\label{eq11-lem11-3}
\\
2N_4\le\ln-\n+s',
\label{eq12-lem11-3}
\\
\nl-\l\ge\frac12s'(s'+1),
\label{eq13-lem11-3}
\\
0\le s'\le N_4.
\label{eq14-lem11-3}
\end{gather}

The frame $((n,n);-\mathbf{e}_2,-\mathbf{e}_1)$ forms small angle with~$Q_2$, 
so by Corollary~\ref{cor3-6}
\begin{equation}\label{eq5-lem11-3}
2N_2\le 2n-\vp-\pv-
\Ceil{\frac{\p-n}{2}} + 1.
\end{equation}

By Proposition~\ref{pr:th5-1},
\begin{gather}
\np-\nl\ge M_1,
\label{eq15-lem11-3}
\\
\pv-\pn\ge M_2,
\label{eq16-lem11-3}
\\
\vp-\vl\ge M_3,
\label{eq17-lem11-3}
\\
\lv-\ln\ge M_4.
\label{eq18-lem11-3}
\end{gather}
Moreover, as the points of $\m$ satisfy~\eqref{eq*-11}, we have
\begin{gather}
\l\ge0,
\label{eq19-lem11-3}
\\
\n\ge0.
\label{eq20-lem11-3}
\end{gather}

Assume that~\eqref{eq1-lem11-3} does not hold.  Then
\begin{equation}\label{eq1'-lem11-3}
2N\ge 4n+6.
\end{equation}

First, assume that~\eqref{eqA-lem11-3} holds.

Let us estimate $2N$ from above.  First, estimate the sum $2N_1+2N_2+2M_2$.  
Using \eqref{eq2-lem11-3}, \eqref{eq5-lem11-3}, \eqref{eq16-lem11-3}, and 
\eqref{eq20-lem11-3}, we have
\begin{equation}
\label{eq201-lem11-3}
2N_1+2N_2+2M_2
\le
2n+s_1-\vp+M_2-\Ceil{\frac{\p-n}{2}} + 1.
\end{equation}
Estimating the ceiling by means of \eqref{eq3-lem11-3}, we obtain:
\begin{equation}
\label{eq202-lem11-3}
\Ceil{\frac{\p-n}{2}}
\ge
-n+\np+s_1+
\Ceil{\frac{n-\np}{2}+\frac14(s_1^2-3s_1)}
.
\end{equation}
It follows from~\eqref{eqA-lem11-3} that $(n-\np)/2\ge0$, and because~$s_1$ is 
an integer, we have
$1/4(s_1^2-3s_1)\ge-1/2$, so we get
\begin{equation*}
\Ceil{\frac{n-\np}{2}+\frac14(s_1^2-3s_1)}
\ge \Ceil{- \frac12} = 0
.
\end{equation*}
Combining this with~\eqref{eq202-lem11-3}, we get
\begin{equation*}
\Ceil{\frac{\p-n}{2}}
\ge
-n+\np+s_1
,
\end{equation*}
and further combining this with~\eqref{eq201-lem11-3} and the inequality $M_2 
\le 1$, we obtain
\begin{equation*}
2N_1+2N_2+2M_2
=3n-\vp-\np+M_2+1\le3n-\vp-\np+2.
\end{equation*}
By means of the last estimate and \eqref{eq6-lem11-3},
\eqref{eq9-lem11-3}, \eqref{eq15-lem11-3}, \eqref{eq17-lem11-3}, and 
\eqref{eq19-lem11-3}, we obtain
\begin{multline*}
2N=(2N_1+2N_2+2M_2)+2N_3+2N_4+2M_1+2M_3+2M_4
\\
\le(3n-\vp-\np+2)+(\vl-\l+s_3)+(\nl-\l+s)+2M_1+2M_3+2M_4
\\
\le3n+2+s_3+s+M_1+M_3+2M_4\le3n+3+s_3+s+M_3+2M_4.
\end{multline*}
Comparing this estimate with \eqref{eq1'-lem11-3}, we get
$$3n+3+s_3+s+M_3+2M_4\ge4n+6,$$
whence
\begin{equation}
\label{eq21-lem11-3}
n\le s_3+s+M_3+2M_4-3.
\end{equation}

As $M\subset\Delta_n$, we have
$$\vp+\v\le2n.$$
Let us estimate the terms on the left-hand side.  Using \eqref{eq19-lem11-3},
\eqref{eq6-lem11-3}, \eqref{eq8-lem11-3}, and \eqref{eq17-lem11-3}, we get
$$\vp=\l+(\vl-\l)+(\vp-\vl)\ge2N_3-s_3+M_3\ge
s_3+M_3.$$
Using \eqref{eq20-lem11-3}, \eqref{eq10-lem11-3}, \eqref{eq18-lem11-3}, and 
\eqref{eq7-lem11-3}, we obtain
$$\v=\n+(\ln-\n)+(\lv-\ln)+(\v-\lv)\ge\frac12s(s+1)+M_4+\frac12s_3(s_3+1).$$
Thus,
$$(s_3+M_3)+\left(\frac12s(s+1)+M_4+\frac12s_3(s_3+1)\right)\le2n,$$
or, equivalently,
\begin{equation}
\label{eq22-lem11-3}
\frac12(s^2_3+3s_3)+\frac12(s^2+s)+M_3+M_4\le2n.
\end{equation}

Now we use \eqref{eq21-lem11-3} to estimate $n$ on the right-hand side of 
\eqref{eq22-lem11-3}:
$$\frac12(s_3^2+3s_3)+\frac12(s^2+s)+M_3+M_4\le2(s_3+s+M_3+3M_4-3).$$
Hence
$$\frac12(s_3^2-s_3)+\frac12(s^2-3s)\le M_3+3M_4-6\le-2,$$
so
$$s^2_3-s_3+s^2-3s\le-4.$$
Completing the squares, we obtain a contradiction:
$$\left(s_3-\frac12\right)^2+\left(s-\frac32\right)^2\le-\frac32,$$
Thus, we have proved \eqref{eq1-lem11-3} provided that \eqref{eqA-lem11-3} 
holds.

Now assume that~\eqref{eqB-lem11-3} holds.

Let us estimate $2N$ from above starting with the sum $2N_1+2N_2+2M_1
+2M_2$.  Using \eqref{eq2-lem11-3}, \eqref{eq5-lem11-3}, \eqref{eq16-lem11-3}, 
and \eqref{eq3-lem11-3}, we obtain
\begin{multline*}
2N_1+2N_2+2M_1+2M_2
\\
\le(\pn-\n+s_1)+
\left(2n-\vp-\pv-\Ceil{\frac{\p-n}{2}} + 1\right)
+2M_1+2M_2
\\
\le
2n-\vp-\n+2M_1+M_2 + 1 -\Ceil{\frac{\np-n}{2} + \frac{s_1^2 - 3s_1}{4}}
\end{multline*}
Estimating $(s_1^2 - 3s_1)/4 \ge - 1/2$, we get
\begin{multline*}
2N_1+2N_2+2M_1+2M_2
\\
\le
2n-\vp-\n+2M_1+M_2 + 1 -\Ceil{\frac{\np-n - 1}{2}}
\\
=
2n-\vp-\n+2M_1+M_2 + 1 -\Floor{\frac{\np-n}{2}}
.
\end{multline*}
Write the estimate in the form
\begin{multline}
\label{eq203-lem11-3}
2N_1+2N_2+2M_1+2M_2
\\
\le
2n-\vp+M_1+M_2 + 1 -
\left(
\Floor{\frac{\np-n}{2}}
+ \n - M_1
\right)
.
\end{multline}
Let us show that
\begin{equation}
\label{eq204-lem11-3}
\Floor{\frac{\np-n}{2}}
+ \n - M_1
\ge 0
.
\end{equation}

Assume that $M_1=1$.  According to~\eqref{eqB-lem11-3}, we have either $\np\ge 
n+2$ or $\np=n+1$.  In the former case we use \eqref{eq20-lem11-3} and 
obtain~\eqref{eq204-lem11-3}.  In the latter case by \eqref{eq15-lem11-3} we 
have $\nl\le\np-M_1=n$, so the edge $[(\nl,\n),(\np,\n)]$ of~$\m$ contains the 
point $(n,\n)$.  Thus, we cannot have $\n=0$, since $\m$ is free of points 
of~$\nz$.  Thus, we must have $\n\ge1$, and \eqref{eq204-lem11-3} follows.

If $M_1=0$, inequality~\eqref{eq204-lem11-3} follows from \eqref{eqB-lem11-3} 
and \eqref{eq20-lem11-3}.

Thus, we have proved~\eqref{eq204-lem11-3} for all possible cases.  Combining 
it with~\eqref{eq203-lem11-3}, we obtain
$$2N_1+2N_2+2M_1+2M_2\le2n-\vp+M_1+M_2+1.$$

Now estimate $2N$ using the last inequality and \eqref{eq6-lem11-3},
\eqref{eq12-lem11-3}, \eqref{eq17-lem11-3}, \eqref{eq18-lem11-3}, 
\eqref{eq19-lem11-3}, and \eqref{eq20-lem11-3}:
\begin{multline*}
2N=(2N_1+2N_2+2M_1+2M_2)+2N_3+2N_4+2M_3+2M_4
\\
\le(2n-\vp+M_1+M_2+1)+(\vl-\l+s_3)
\\
+(\ln-\n+s')+2M_3+2M_4
\\
\le
2n+1+s_3+s'+(M_3-(\vp-\vl))
\\
+(\ln+M_4)+M_1+M_2+M_3+M_4
\\
\le 2n+1+s_3+s'+\lv+M_1+M_2+M_3+M_4.
\end{multline*}
Comparing this estimate with \eqref{eq1'-lem11-3} we obtain
\begin{equation}
\label{eq23-lem11-3}
\lv\ge-s_3-s'+2n+5-M_1-M_2-M_3-M_4.
\end{equation}
Together with \eqref{eqB-lem11-3} this implies
\begin{equation}
\label{eq24-lem11-3}
n\le s_3+s'-5+M_1+M_2+M_3+M_4.
\end{equation}

The triangle $\Delta_n$ lies in the half-plane $x_1\le2n$, so we have
\begin{equation}
\label{eq205-lem11-3}
\np\le2n.
\end{equation}
We can estimate the left-hand side by means of~\eqref{eq19-lem11-3}, 
\eqref{eq13-lem11-3}, and~\eqref{eq15-lem11-3} as follows:
$$\np=\l+(\nl-\l)+(\np-\nl)\ge\frac12s'(s'+1)+M_1.$$
Using this estimate and~\eqref{eq24-lem11-3}, we obtain 
from~\eqref{eq205-lem11-3}:
$$\frac12s'(s'+1)+M_1\le2s_3+2s_4-10+2M_1+2M_2+2M_3+2M_4,$$
whence
$$2s_3\ge\frac12(s'^2-3s')+10-M_1-2M_2-2M_3-2M_4\ge$$
$$\ge\frac12(s'^2-3s')+3=\frac12(s'^2-3s'+6),$$
and finally
\begin{equation}\label{eq25-lem11-3}
s_3\ge\frac14(s'^2-3s'+6).
\end{equation}

The vertices of $\m$ solve \eqref{eq*-11}, so
\begin{equation}
\label{eq206-lem11-3}
\vp+\v\le2n.
\end{equation}
Using \eqref{eq19-lem11-3}, \eqref{eq6-lem11-3}, \eqref{eq8-lem11-3}, 
and~(\ref{eq17-lem11-3}), we deduce
$$\vp=\l+(\vl-\l)+(\vp-\vl)\ge2N_3-s_3+M_3\ge s_3+M_3,$$
while by virtue of~\eqref{eq23-lem11-3} and \eqref{eq7-lem11-3} we obtain
$$\v=\lv+(\v-\lv)\ge(-s_3-s'+2n+5-M_1-M_2-M_3-M_4)+\frac12s_3(s_3+1).$$
Combining~\eqref{eq206-lem11-3} with two last estimates, we get
$$s'\ge\frac12s_3(s_3+1)+5-M_1-M_2-M_4\ge\frac12s_3(s_3+1)+2=\frac12(s^2_3+s_3+4),$$
and finally
\begin{equation}\label{eq26-lem11-3}
s'\ge\frac12(s^2_3+s_3+4).
\end{equation}

It is not hard to check that the inequalities \eqref{eq25-lem11-3} 
and~\eqref{eq26-lem11-3} are incompatible.  This contradiction proves 
\eqref{eq1-lem11-3} in case~\eqref{eqB-lem11-3} holds.
\end{proof}

\begin{lemma}\label{lem11-4}
Suppose that the vertices of a type~Va$_n$ $N$-gon $\m$ belong to a 
$(1,n)$-lattice $\Gamma$; then
\begin{equation}\label{eq1-lem11-4}
N\le2n-2.
\end{equation}
\end{lemma}
\begin{proof}
Note that by Proposition~\ref{pr:th1-11} the product of the small 
$\mathbf{e}_1$-step and the large $\mathbf{e}_2$-step of~$\Gamma$ equals~$n$.

First, assume that $\Gamma$ has small $\mathbf{e}_1$-step $s\ge2$.  In this 
case $s$ divides $n$ and all the points of $\Gamma$ belonging to~$\Delta_n$ 
lie on the lines
$$x_1=js\qquad(j=0,\dots,2n/s).$$
Consequently, the vertices of $\m$ lie on the same lines as well. Each of the 
$2n/s$ lines corresponding to $j=0,\dots,2n/s-1$ contains at most two vertices 
while the line $x_1=2n$ corresponding to $j=2n/s$ does not contain any vertex, 
since its only common point with $\Delta_n$ is $(2n,0)\in\nz$.  Thus, if 
$s\ge3$, we can estimate the number of vertices of $\m$ as follows:
$$N\le2\cdot\frac{2n}{s}\le\frac{4n}{3}\le2n-2$$
(since $n\ge3$).  If $s=2$, the large $\mathbf{e}_2$-step of~$\Gamma$ is 
$n/2$, so each of the lines the lines $x_1=0$ and $x_1=n$ has a single point 
of $\Gamma$ between adjacent points of $\nz$.  Consequently, each of these 
lines (corresponding to $j=0$ and $j=n/2$) contains at most one vertex of $\m$ 
and there are $n-2$ other lines containing up to two vertices.  Thus, the 
total number of vertices does not exceed
$$N\le2+2\cdot(n-2)=2n-2.$$
We have thus proved the lemma under the hypothesis that the small 
$\mathbf{e}_1$-step of~$\Gamma$ is greater then~1.

Assume that $\Gamma$ has small $\mathbf{e}_1$-step~1.  Then by 
Proposition~\ref{pr:th1-12} the lattice $\Gamma$ admits a basis of the form 
$(\mathbf{e}_1-a\mathbf{e}_2, n\mathbf{e}_2)$, where $\floor{n/2}-n+1\le 
a\le\floor{n/2}$.  Consider the linear transformation
$$
A=
\begin{pmatrix}
1 & 0 \\
a/n & 1/n
\end{pmatrix}
.$$
It is easily seen that $A\Gamma=\zz$ and $\Lambda=A(\nz)$ is a lattice having 
the basis $(n\mathbf{e}_1, \mathbf{e}_2)$; the image~$\tm$ of $\m$ is an 
integer $N$-gon free from points of $\Lambda$ and contained in the triangle 
$\widehat{\Delta}=A\Delta_n$ with the vertices $\mathbf 0$, $(0, 2)$, and 
$(2n, 2a)$.

Set
\begin{align}
p & =\min\{x_2\colon(x_1,x_2)\in\widehat{\Delta}\} = \min\{2a,0\}
\label{eq2-lem11-4}
,
\\
q & =\max\{x_2\colon(x_1,x_2)\in\widehat{\Delta}\} = \max\{2a,2\}
\label{eq2a-lem11-4}
.
\end{align}

Note two obvious facts.  Firstly, all the integer points of the 
triangle~$\widehat{\Delta}$ lie on the $q-p+1$ lines
\begin{equation}\label{eq3-lem11-4}
x_2=j\qquad(j=p,p+1,\ldots,q),
\end{equation}
so \emph{all the vertices of~$\tm$ lie on the lines \eqref{eq3-lem11-4}, each 
line containing at most two vertices}.  Secondly, \emph{if $a\ne 0$, there are 
no vertices of $\tm$ on the line $x_2=p$}, as it is easy to check that in this 
case the line has the only common point with $\widehat{\Delta}$---the vertex 
of the triangle, which belongs to~$\Lambda$.  By the same argument, \emph{if 
$a\ne 1$, there are no vertices of $\tm$ on the line $x_2=q$}.

As a consequence, we see that \eqref{eq1-lem11-4} holds, provided that
\begin{equation}\label{eq4-lem11-4}
q-p\le n.
\end{equation}
Indeed, if additionally $a\ne 0$ and $a\ne 1$, then the vertices of~$\tm$ lie 
on the $ q-p-1 $ lines \eqref{eq3-lem11-4} corresponding to $j=p+1,\dots,q-1$, 
whence
$$N\le2(q-p-1)\le2n-2.$$
On the other hand, if $a=0$ or $a=1$, then according to \eqref{eq2-lem11-4} 
and~\eqref{eq2a-lem11-4} we have $p=0$, $q=2$, and the vertices of~$\tm$ lie 
on the lines \eqref{eq3-lem11-4} (excluding $j=p$ and $j=q$), whence
$$N\le4\le2n-2$$
as $n \ge 3$.

According to \eqref{eq2-lem11-4} and~\eqref{eq2a-lem11-4}, we have
$$q-p=
\begin{cases}
2a,& \text{if } a\ge1,
\\
-2a+2,&\text{if } a\le0.
\end{cases}
$$
Using this formula, we see that if $1\le
a\le\frac{n}{2}$, then
$$q-p\le2\cdot\frac{n}{2}=n,$$
and if $n/2-n+1\le a\le0$, then
$$q-p\le-2(n/2-n+1)+2=n.$$
Thus, inequality \eqref{eq4-lem11-4} holds provided that $n/2-n+1\le
a\le n/2$, so for these~$a$ inequality~\eqref{eq1-lem11-4} is proved.  Taking 
into account the range of possible values of $a$, we see that it only remains 
to check the case $a=\floor{n/2}-n+1$ when $\floor{n/2}<n/2$, i.~e.\ 
$a=(-n+1)/2$ for an odd~$n$.

In this case the vertices of $\widehat{\Delta}$ are the points $\mathbf 0$, 
$(0, 2)$, and $(2n, -n + 1)$, and this triangle is the solution set of the 
system
\begin{equation}
\label{eq5-lem11-4}
\left\{
\begin{array}{l}
x_1\ge0,\\
x_1\le-\frac{2n}{n+1}(x_2-2),\\
x_1\ge-\frac{2n}{n-1}x_2.
\end{array}
\right.
\end{equation}
We will show that if $n\ge5$, each of the lines $x_2=1$ and $x_2=-n+2$ 
contains at most one vertex of~$\tm$.

It follows from system~\eqref{eq5-lem11-4} that points of~$\widehat{\Delta}$ 
lying on the line $x_2=1$ satisfy
$$0\le x_1\le\frac{2n}{n+1}.$$
As $2n/(n+1)<2$ and all the integer points of the line $x_1=0$ belong to 
$\Lambda$, we see that the line $x_2 = 1$ has only one point that could be a 
vertex of~$\tm$: the point~$(1,1)$.

It also follows from~\eqref{eq5-lem11-4} that the points of~$\widehat{\Delta}$ 
lying on the line $x_2=-n+2$ satisfy
$$\frac{2n(n-2)}{n-1}\le x_1\le\frac{2n^2}{n+1}.$$
Given that $n\ge5$, we have
$$2n-3<\frac{2n(n-2)}{n-1}<\frac{2n^2}{n+1}<2n-1,$$
so $(2n-2,-n+2)$ is the only point of the line $x_2=-n+2$ that could be a 
vertex of~$\tm$.

As in our case $p = -n+1$ and $q = 2$, we see that each of the $n-2$ 
lines~\eqref{eq3-lem11-4} corresponding to $j=p+2,\dots,q-2$ contains at most 
two vertices of $\tm$; each of the lines corresponding to $j=p+1$ and $j=q-1$ 
contains at most one vertex; finally, as $a\neq0$ and $a\neq1$, the lines 
corresponding to $j=p$ and $j=q$ contain no vertices.  This amounts to a total 
of at most $2n - 2$ vertices, so \eqref{eq1-lem11-4} is proved.

It only remains to check the case $n=3$.  Then the vertices of the 
triangle~$\widehat{\Delta}$ are the points $\mathbf{0}$, $(0,2)$, and 
$(6,-2)$.  It is easy to check that $\widehat \Delta$ contains only 4 integer 
points not belonging to $\Lambda$, so the inequality $N \le 4 = 2n - 2$ is 
trivial.
\end{proof}

\begin{definition}
We call an integer polygon \emph{minimal} if it does not contain other integer 
polygon with the same number of vertices.
\end{definition}

We note two simple properties of minimal polygons.

\begin{proposition}
\label{pr:min1}
Any edge of a minimal polygon contains precisely two interger points---its 
endpoints.
\end{proposition}
\begin{proof}
If $\mathbf{v}_1$,\dots,$\mathbf{v}_N$ are the vertices of an integer $N$-gon 
and its edge $[\mathbf{v}_1,\mathbf{v}_2]$ contains an integer point 
$\mathbf{v}$ different from $\mathbf{v}_1$ and $\mathbf{v}_2$, then it is 
easily seen that the convex hull of the points $\mathbf{v}$, $\mathbf{v}_2$, 
\dots,$\mathbf{v}_n$ is an integer $N$-gon contained in the original one and 
different from it.  This means that the original polygon is not minimal.
\end{proof}

\begin{proposition}
\label{pr:min2}
Affine automorphisms of the integer lattice map minimal polygons onto minimal 
polygons.
\end{proposition}

This proposition is obvious.

\begin{proof}[Proof of Lemma~\ref{lem:th11-2}]
Let us prove the inequality
\begin{equation}\label{eq11-5-1}
N\le2n+2.
\end{equation}
We can certainly assume that $\m$ is minimal, for if not, we replace $\m$ by a 
minimal polygon it contains (which is, of course, again a type~Va$_n$ 
polygon).

First, assume that $\m$ satisfies either
\begin{equation}\label{eq11-5-2}
\np\le n
\end{equation}
or
\begin{equation}\label{eq11-5-3}
\np\ge n+1,\quad \lv\le n.
\end{equation}

If $\m$ lies in a slab of the form
$$0\le x_1\le n,\quad n\le x_1\le 2n,\quad 0\le x_2\le n,\quad n\le x_2\le 2n,$$
it is a type~I$_n$ polygon, and the estimate \eqref{eq11-5-2} follows from 
Theorem~\ref{subtheorem-c} for type~I polygons.  Otherwise, $\m$ is split by 
the segments $[(n,0),(n,n)]$ and $[(0,n),(n,n)]$, which are the intersections 
of the lines $x_1=n$ and $x_2=n$ with~$\Delta_n$.  Therefore, by 
Proposition~\ref{pr:th5-3} the frame $((n,n);-\mathbf{e}_2,-\mathbf{e}_1)$ 
splits the slope~$Q_2$.  If this frame forms small angle with the slope, 
Lemma~\ref{lem11-3} provides~\eqref{eq11-5-1}.  If not, it follows from 
Proposition~\ref{pr:sa} that the frame $((n,n);-\mathbf{e}_1,-\mathbf{e}_2)$ 
forms small angle with $Q_2$.  Let $\m'$ be the reflection of $\m$ about the 
line $x_1 = x_2$.  It is not hard to check that 
$((n,n);-\mathbf{e}_2,-\mathbf{e}_1)$ forms small angle with $Q_2(\m')$; 
moreover, $\m'$ is a minimal type~Va$_n$ polygon, and since
$$\np(\m')=\lv(\m),\qquad \lv(\m')=\np(\m),$$
we see that $\m'$ satisfies \eqref{eq11-5-2} or \eqref{eq11-5-3}.  Applying 
the already proved part of the lemma to $\m'$, we obtain \eqref{eq11-5-1}.

Now suppose that $\m$ satisfies neither~\eqref{eq11-5-2}, 
nor~\eqref{eq11-5-3}.  Thus, in particular,
$$\np(\m)\ge n+1.$$
Consider the affine automorphism of $\nz$ given by
$$\varphi(x_1,x_2)=(-x_1-x_2+2n,x_2).$$
By Proposition~\ref{pr:min2}, the polygon $\varphi(\m)$ is minimal.  Moreover, 
it lies in the triangle $\Delta_n$, since $\Delta_n=\varphi(\Delta_n)$.  
Obviously, we have
$$\n(\varphi(\m))=\n(\m).$$
A straightforward computation gives
$$\nl(\varphi(\m))=-\np(\m)-\n(\m)+2n.$$
As the points of $\m$ satisfy system~\eqref{eq*-11}, we have $\n(\m)\ge0$, so
$$\nl(\varphi(\m))\le-\np(\m)+2n\le n-1.$$
By Proposition~\ref{pr:min1},
$$\np(\varphi(\m))\le\nl(\varphi(\m))+1,$$
so we have
$$\np(\varphi(\m))\le n.$$
Thus, the polygon~$\varphi(\m)$ satisfies~\eqref{eq11-5-2}, and applying the 
already proved part of the lemma to $\varphi(\m)$, we obtain~\eqref{eq11-5-1}.

The part of the lemma concerning polygons with vertices belonging to a 
$(1,n)$-lattice is given by Lemma~\ref{lem11-4}.
\end{proof}

\section{Type VI polygons}
\label{sec:vi}

It turns out that any type~VI polygon can be mapped onto a polygon of another 
type by an automorphism of~$\nz$.  The following lemma is the main result of 
this section.

\begin{lemma}
\label{lem:th12-1}
Suppose that~$\m$ is a type~VI$_n$ polygon; then there exists an affine 
automorphism~$\psi$ of~$\nz$ such that $\psi(\m)$ is a polygon of one of the 
types I$_n$, II$_n$, III$_n$, or~V$_n$.
\end{lemma}
\begin{proof}
The polygon~$\m$ is split by the segments $[\mathbf{0},(-n,0)]$ and 
$[\mathbf{0},(0,n)]$, so the lift~$\tm$ (see Section~\ref{sec:lift}) is 
well-defined.  The polygon~$\tm$ is split by the segment $[\mathbf{0},(-n,0)]$ 
by the definition of the lift and by the segment $[\mathbf{0},(0,n)]$ by 
Lemma~\ref{lem11-02}.  By the same lemma, the segment $[\mathbf{0},(-n,-n)]$ 
does not split~$\tm$.  The lines~$x_1=\pm n$ are invariant under the lift 
transformation, so they do not split~$\tm$ either.  Besides, $\tm$ has points 
in the slab
\begin{equation}\label{eq1-th12-1}
-n\le x_1\le n,
\end{equation}
(e.~g.\ on the segment $[\mathbf{0},(0,n)]$), so we conclude that it is 
contained in this slab.

If the line $x_2=n$ does not split~$\tm$, the latter is a type~V$_n$ polygon. 
Otherwise, $\tm$ is split by one of the segments $[(-n,n), (0,n)]$ and 
$[(0,n),(n,n)]$, since their union is exaclty the set of common points of the 
line $x_2=n$ and the slab \eqref{eq1-th12-1}.

Suppose that the segment $[(-n,n),(0,n)]$ splits~$\tm$.  Let $T$ be the 
translation by the vector $(n,0)\in\nz$.  The polygon~$T\tm$ is split by the 
segments
\begin{gather*}
[(0,n),(n,n)]=T[(-n,n),(0,n)],
\\
[(n,0),\mathbf{0}]=T[\mathbf{0},(-n,0)],
\\
[(n,0),(n,n)]=T[\mathbf{0},(0,n)]
\end{gather*}
and is not split by the line $x_1=0$ being the image of $x_1=-n$ under~$T$.   
In other words, $T\tm$ is a type~III$_n$ polygon, and we are done.

It remains to consider the case of the segment $[(0,n),(n,n)]$ 
splitting~$\tm$.  Let~$\varphi$ be an affine automorphism of~$\nz$ defined by
$$\varphi(x_1,x_2)=(-x_1,n-x_2),$$
(the symmetry with respect to $(0, n/2)$) and set $\m'=\varphi(\tm)$.  The 
polygon~$\m'$ lies in the slab~\eqref{eq1-th12-1}, which is invariant 
under~$\varphi$; also~$\m'$ is split by the segments
\begin{gather*}
[(-n,0),\mathbf{0}]=\varphi([(n,n),(0,n)]),
\\
[\mathbf{0},(0,n)]=\varphi([(0,n),\mathbf{0}])
\end{gather*}
and is not split by the segment
$$[(0,n),(n,2n)]=\varphi([\mathbf{0},(-n,-n)]).$$
Thus, the lift~$\tm'$ is well-defined.  By the definition of the lift and by 
Lemma~\ref{lem11-02}, the polygon~$\tm'$ is split by the segments 
$[\mathbf{0},(-n,0)]$ and $[\mathbf{0},(0,n)]$ and is not split by the 
segments $[\mathbf{0},(-n,-n)]$ and $[(0,n),(n,2n)]$; moreover, $\tm'$ lies in 
the slab~\eqref{eq1-th12-1}.  Consequently, the line $x_1 = -n$ does not 
split~$\tm'$.  If the line~$x_2=n$ does not split it either, it is a 
type~V$_n$ polygon, and we are done.  Otherwise, as before, we infer that 
either $[(-n,n),(0,n)]$ splits~$\tm'$, and we conclude by noticing that 
$T\tm'$ is a type~III$_n$ polygon, or $[(0,n),(n,n)]$ splits~$\tm'$, which we 
assume in what follows.

The intersection of the line $x_1-x_2=-n$ and the slab~\eqref{eq1-th12-1} is 
the union of the segments $[(-n,0),(0,n)]$ and $[(0,n),(n,2n)]$.  The latter 
segment does not split~$\tm'$, the line $x_1-x_2=-n$ splits~$\tm'$ if and only 
if the segment $[(-n,0),(0,n)]$ does so.  Likewise, the line $x_1-x_2=0$ 
splits~$\tm'$ if and only if the segment $[\mathbf{0},(n,n)]$ does so, for 
$\tm'$ is not split by $[(-n,-n),\mathbf{0}]$.  Thus, we have four logical 
possibilities.

\emph{Case~1.} The lines $x_1-x_2=-n$ and $x_1-x_2=0$ do not split~$\tm'$.

\emph{Case~2.} The segments $[(-n,0),(0,n)]$ and $[\mathbf{0},(n,n)]$ 
split~$\tm'$.

\emph{Case~3.} The segment $[(-n,0),(0,n)]$ splits $\tm'$ and the line 
$x_1-x_2=0$ does not.

\emph{Case~4.} The segment $[\mathbf 0,(n,n)]$ splits $\tm'$ and the line 
$x_1-x_2=-n$ does not.

In Case~1 define the affine automorphism of~$\nz$ by
$$\psi_1(x_1,x_2)=(-x_1+x_2,x_2)$$
and consider the polygon $\psi_1(\tm')$.  It is not split by the lines $x_1=0$ 
and $x_1=n$, being the images of $x_1-x_2=0$ and $x_1-x_2=-n$, respectively, 
and $\psi_1(\tm')$ has points inside the slab $0\le x_1\le n$, e.~g.\ on the 
segment
$$[(n,0),(0,n)]=\psi_1([\mathbf 0,(0,n)]).$$
Consequently, $\psi_1(\tm')$ is a type~I$_n$ polygon.

In Cases~2 and~3 we use the same automorphism~$\psi_1$.  It is not hard to 
check that in Case~2, $\psi_1(\tm')$ is a type~II$_n$ polygon and in Case~3, 
it is a type~III$_n$ polygon

In Case~4, define the automorphism of~$\nz$ by
$$\psi_2(x_1,x_2)=(x_1-x_2+n,x_2).$$
It is easily seen that $\psi_2(\tm')$ is a type~III$_n$ polygon.
\end{proof}

\section{Proof of Theorem~\ref{subtheorem-c} for polygons of types V and VI}
\label{sec:last}

We have already proved Theorem~\ref{subtheorem-c} for polygons of 
types~I$_n$--IV$_n$.  Type~Va$_n$ polygons have been our stumbling block so 
far, because we have not proved the estimate $N \le 2n$ for such $N$-gons, in 
case their vertices belong to a $(1, n/2)$-lattice.  However, for type~Va$_n$ 
$N$-gons we have the estimate $N \le 2n + 2$ (Lemma~\ref{lem:th11-2}).  In 
view of Lemmas~\ref{lem:th11-1} and~\ref{lem:th12-1}, this estimate is valid 
for type~V$_n$ and~VI$_n$ $N$-gons as well.  Combining this with 
Theorem~\ref{th:types}, we obtain the following particular case of 
Theorem~\ref{th:mt}:
\begin{lemma}
\label{lem:sq}
Let~$n$ be an integer, $n \ge 3$; then any integer polygon free of 
$\nz$-points has no more than $2n + 2$ vertices.
\end{lemma}

Now we can prove the missing estimate for type~Va polygons.

\begin{lemma}
\label{lem13-2}
Let~$n$ be an even integer, $n \ge 4$, and~$\m$ be a type~Va$_n$ $N$-gon.  
Suppose that the vertices of~$\m$ belong to a $(1, n/2)$-lattice.  Then
\begin{equation}\label{eq1-lem13-2}
N\le 2n.
\end{equation}
\end{lemma}
\begin{proof}
Let $\Gamma$ be a $(1, n/2)$-lattice containing the vertices of~$\m$.

By definition, $\m$ is contained in the triangle~$\Delta_n$ defined 
by~\eqref{eq*-11}.

Assume that $n=4$, then we must prove that
\begin{equation}
\label{eq2-lem13-2}
N\le8.
\end{equation}
In this case $\Gamma$ is a $(1,2)$-lattice, and by 
Proposition~\ref{pr:th1-11}, the only possible values for the small 
$\mathbf{e}_1$- and $\mathbf{e}_2$-steps of~$\Gamma$ are~1 and~2.

If the small $\mathbf{e}_1$-step of~$\Gamma$ is~2, it is easily seen that all 
the points of~$\Gamma$ lying in~$\Delta_4$ and not belonging to $4\zz$ lie on 
the lines
$$x_1=0,\quad x_1=2,\quad x_1=4,\quad x_1=6.$$
In particular, all the vertices of~$\m$ lies on these lines.  Each of the four 
lines contains at most two vertices, and~\eqref{eq2-lem13-2} follows.

The case when the small $\mathbf{e}_2$-step of~$\Gamma$ is~2 is handled in the 
same way.

Suppose that the small $\mathbf{e}_1$-step of~$\Gamma$ is~1.  By 
Proposition~\ref{pr:th1-12} it has a basis $(\mathbf{e}_1+a\mathbf{e}_2, 
2\mathbf{e}_2)$, where $a=0$ or $a=1$. In the former case the small 
$\mathbf{e}_1$-step of~$\Gamma$ is~2 and \eqref{eq2-lem13-2} is proved.  In 
the latter case define the automorphism of~$4\zz$ by
$$\psi(x_1,x_2)=(x_1,x_1-x_2+8).$$
The polygon $\psi(\m)$ lies in the triangle $\Delta_4=\psi(\Delta_4)$ and its 
vertices belong to a $(1, 2)$-lattice~$\psi(\Gamma)$ spanned by 
$(\mathbf{e}_1, -2\mathbf{e}_2)$.  Obviously, the small $\mathbf{e}_2$-step 
of~$\psi(\Gamma)$ is~2, so we obtain~\eqref{eq2-lem13-2} applying the proved 
part of the lemma to~$\psi(\tm)$.

Now assume that $n \ge 6$ and, contrary to our assertion,
\begin{equation}\label{eq3-lem13-2}
N>2n.
\end{equation}
It follows from Proposition~\ref{pr:snf} that there exists a unimodular 
transformation~$B$ such that $B\Gamma = \z \times (n/2)\z$.  Clearly, the 
transformation $A=\diag(1,2/n)B$ maps~$\Gamma$ onto~$\zz$, so $\m'=A\m$ is an 
integer $N$-gon contained in the integer triangle $\Delta'=A\Delta_n$.

Let us estimate the number of integer points in~$\m'$, which we denote 
by~$n_0$.

We claim that
\begin{equation}\label{eq4-lem13-2}
n_0\ge(n-1)^2.
\end{equation}

Indeed, if $\m'$ contained less then $(n-1)^2$ integer points, we could find 
two integers $i_1$ and $i_2$ such no integer point $(u_1,u_2)\in\m'$ would 
satisfy
\begin{gather*}
x_k\equiv i_k\pmod{(n-1)}
\end{gather*}
simultaneously for $k = 0$ and $k = 1$.  In other words, the shifted $N$-gon
$\m' - (i_1, i_2)$ would be free of $(n-1)\zz$-points, which is impossible due 
to Lemma~\ref{lem:sq} and inequality~\eqref{eq3-lem13-2}.  
Thus,~\eqref{eq4-lem13-2} is proved.

To estimate $n_0$ from above we use the inclusion $\m'\subset\Delta'$.  
Let~$n'$ be the number of integer points in the interior of~$\Delta'$ and $b'$ 
be the number of integer points on its boundary.  Observe that the vertices 
and midpoints of the sides of~$\Delta_n$ are $\nz$-points, so they do not 
belong to~$\m$; consequently, the vertices and the midpoints of~$\Delta'$ do 
not belong to~$\m'$.  Therefore, if~$\m'$ has common points with a side 
of~$\Delta'$, they lie between the midpoint of the side and one of its 
endpoints.  This and the fact that the vertices of~$\Delta'$ are integer 
points, imply
$$n_0<n'+\frac{b'}{2}.$$

Let~$s'$ be the area of~$\Delta'$.  By Pick's theorem,
\begin{equation*}
s' = n' + \frac{b'}{2} - 1
,
\end{equation*}
whence
$$n_0<s'+1.$$
On the other hand, $s' = 4n$, because the area of~$\Delta_n$ is $2n^2$ and 
$|\det A|=2/n$.  Finally, we obtain
$$n_0<4n+1.$$

Comparing the last inequality and~\eqref{eq4-lem13-2}, we get
$$(n-1)^2<4n+1.$$
which cannot hold if $n \ge 6$.  The contradiction proves the lemma for $n \ge 
6$.
\end{proof}

As a corollary of Lemmas~\ref{lem13-2} and Lemma~\ref{lem:th11-2} we obtain 
that Theorem~\ref{subtheorem-c} holds for type~Va polygons.  In view of 
Lemma~\ref{lem:th11-1} we conclude that it also holds for type~V polygons.  
Now, Lemma~\ref{lem:th12-1} implies that Theorem~\ref{subtheorem-c} is true 
for type~VI polygons as well.

Thus, Theorem~\ref{subtheorem-c} is proved in full generality.  Combining it 
with other results of~\cite{BKa}, we conclude that Theorem~\ref{th:mt}---the 
Main Theorem of~\cite{BKa}---is proved as well.

\bibliographystyle{plain}
\bibliography{lit}

\end{document}